\documentclass[one column]{article}
\usepackage{amsmath,amssymb,amsthm,color,enumerate,fullpage,mathrsfs,graphicx,subcaption}
\usepackage[nameinlink]{cleveref}

\usepackage[noblocks,auth-sc]{authblk}
\usepackage[sort&compress,numbers]{natbib}

\definecolor{lawngreen}{RGB}{0,250,154}
\definecolor{pink}{RGB}{250,0,154}
\definecolor{gray}{RGB}{192,192,192}

\newtheorem{theorem}{Theorem}

\newtheorem{lemma}[theorem]{Lemma}
\newtheorem{corollary}[theorem]{Corollary}

\newtheorem{remark}[theorem]{Remark}

\theoremstyle{definition}
\newtheorem{Def}[theorem]{Definition}

\crefname{assumption}{Assumption}{Assumptions}
\crefname{lemma}{Lemma}{Lemmata}
\crefname{theorem}{Theorem}{Theorems}
\crefname{corollary}{Corollary}{Corollaries}
\crefname{prop}{Proposition}{Propositions}
\crefname{claim}{Claim}{Claims}
\crefname{procedure}{Procedure}{Procedures}
\crefname{figure}{Figure}{Figures}
\crefname{remark}{Remark}{Remarks}
\crefname{section}{Section}{Sections}
\crefname{procedure}{Procedure}{Procedures}
\crefname{table}{Table}{Tables}
\crefname{Def}{Definition}{Definition}

\allowdisplaybreaks

\newcommand{\R}{\mathbb{R}}
\newcommand{\Z}{\mathbb{Z}}
\newcommand{\Q}{\mathbb{Q}}
\newcommand{\eps}{\varepsilon}
\renewcommand{\tilde}{\widetilde}
\renewcommand{\hat}{\widehat}
\renewcommand{\bar}[1]{\overline{#1}}

\newcommand{\DMILP}{\mathscr T^{D-MI}}
\newcommand{\MILP}{\mathscr T^{MI}}
\newcommand{\DMILPR}{\DMILP_R }
\newcommand{\MILPR}{\MILP_R}
\newcommand{\GDMILP}{\mathscr T^{\hat{D-MI}}}
\newcommand{\GMILP}{\mathscr T^{\hat{MI}}}
\newcommand{\GDMILPR}{\GDMILP_R }
\newcommand{\GMILPR}{\GMILP_R}

\newcommand{\inte}{\operatorname{int}}

\newcommand{\relintr}{\operatorname{relint}}

\newcommand{\relint}{\relintr}
\newcommand{\rec}{\operatorname{rec}}

\newcommand{\aff}{\operatorname{aff}}
\newcommand{\cl}{\operatorname{cl}}

\newcommand{\proj}{\operatorname{proj}}
\newcommand{\cone}{\operatorname{cone}}

\newcommand{\floor}[1]{\left\lfloor#1\right\rfloor}
\newcommand{\ceil}[1]{\left\lceil#1\right\rceil}

\title{Mixed-integer bilevel representability}
\author[1]{Amitabh Basu}
\author[2]{Christopher Thomas Ryan}
\author[3]{Sriram Sankaranarayanan\thanks{Email: ssankar5@jhu.edu}}
\affil[1]{Department of Applied Mathematics and Statistics, Johns Hopkins University}
\affil[2]{Booth School of Business, University of Chicago}
\affil[3]{Department of Civil Engineering, Johns Hopkins University}

\date{}

\begin{document}
\maketitle
\begin{abstract}
	{We study the representability of sets that admit extended formulations using mixed-integer bilevel programs.} We show that feasible regions modeled by continuous bilevel constraints (with no integer variables), complementarity constraints, and polyhedral reverse convex constraints are all finite unions of polyhedra. Conversely, any finite union of polyhedra can be represented using any one of these three paradigms. We then prove that the feasible region of bilevel problems with integer constraints exclusively in the upper level is a finite union of sets representable by mixed-integer programs and vice versa. Further, we prove that, up to topological closures, we do not get additional modeling power by allowing integer variables in the lower level as well. To establish the last statement, we prove that the family of sets that are finite unions of mixed-integer representable sets forms an algebra of sets (up to topological closures). 
\end{abstract}

\section{Introduction}\label{s:introduction}


This paper studies \emph{mixed-integer bilevel linear} (MIBL) programs of the form
\begin{equation}\label{eq:intro-mibp}
\begin{aligned}
\max_{x,y} \ \ & c^\top x + d^\top y \\ 
\text{ s.t. } & Ax + By \le b \\
              & y \in \arg\max_y \{ f^\top y : Cx + Dy \le g, y_i \in \Z \text { for } i \in \mathcal I_F\} \\
              & x_i \in \Z \text{ for } i \in \mathcal I_L 
\end{aligned}
\end{equation}
where $x$ and $y$ are finite-dimensional real decision vectors, $b$, $c$, $d$, $f$ and $g$ are finite-dimensional vectors and the constraint matrices $A$, $B$, $C$, and $D$ have conforming dimensions. The decision-maker who determines $x$ is called the \emph{leader}, while the decision-maker who determines $y$ is called the \emph{follower}. The sets $\mathcal I_L$ and $\mathcal I_F$ are subsets of the index sets of $x$ and $y$ (respectively) that determine which leader and follower decision variables are integers. The follower solves the \emph{lower level problem} 
\begin{equation}\label{eq:intro-lower-level}
\begin{aligned}
\max_{y} \ \ & f^\top y \\ 
\text{ s.t. } & Dy \le g - Cx \\
              & y_i \in \Z \text{ for } i \in \mathcal I_F 
\end{aligned}
\end{equation}
for a given choice of $x$ by the leader.  

Bilevel programming has a long history, with traditions in theoretical economics (see, for instance, \citep{mirrlees1999theory}, which originally appeared in 1975) and operations research (see, for instance, \citep{candler1982linear,jeroslow1985polynomial}). While much of the research community's attention has focused on the continuous case, there is a growing literature on bilevel programs with integer variables, starting with early work in the 1990s by Bard and Moore \citep{moore1990mixed,bard1992algorithm} through a more recent surge of interest \citep{Lozano2017,Wang2017,Fischetti2017,Fischetti2017a,dempe2017solving,yue2017projection,tahernejad2016branch,Lodi2014,schaefer2018solving}. Research has largely focused on algorithmic concerns, with a recent emphasis on leveraging advancements in cutting plane techniques. Typically, these algorithms restrict how variables appear in the problem. For instance, \citet{Wang2017} consider the setting where all variables are integer-valued. \citet{Fischetti2017a} allow for continuous variables but restrict the leader's continuous variables from entering the follower's problem. Only very few papers have studied questions of computational complexity in the mixed-integer setting, and also often with restricting the appearance of integer variables (see, for instance, \citep{Koppe2010}). 

To our knowledge, a thorough study of general MIBL programs with no additional restrictions on the variables and constraints has not been undertaken in the literature. The contribution of this paper is to ask and answer a simple question: what types of sets can be modeled as feasible regions (or possibly projections of feasible regions) of such general MIBL programs? Or put in the standard terminology of the optimization literature: what sets are \emph{MIBL-representable}? 

Separate from the design of algorithms and questions of computational complexity, studying representability shows the reach of a modeling framework. The classical paper of \citet{Jeroslow1984} provides a characterization of sets that can be represented by mixed-integer linear feasible regions. They show that a set is the projection of the feasible region of a mixed-integer linear problem (termed \emph{MILP-representable}) if and only if it is the Minkowski sum of a finite union of polytopes and a finitely-generated integer monoid (concepts more carefully defined below). This result is the gold standard in the theory of representability, as it answers a long-standing question on the limits of mixed-integer programming as a modeling framework. Jeroslow and Lowe's result also serves as inspiration for recent interest in the representability of a variety of problems. See \citet{vielma2015mixed} for a review of the literature until 2015 and \citep{lubin2017mixed,lubin2017regularity,Basu2016,lubin2016extended,del2017ellipsoidal,del2017characterizations} for examples of more recent work. 

To our knowledge, questions of representability have not even been explicitly asked of  \emph{continuous bilevel linear} (CBL) programs where $\mathcal I_L = \mathcal I_F = \emptyset$ in \eqref{eq:intro-mibp}. Accordingly, our initial focus concerns characterizations of CBL-representable sets. In the first key result of our paper (\cref{thm:RepEquiv}), we show that every CBL-representable set can also be modeled as the feasible region of a linear complementarity (LC) problem (in the sense of \citet{Cottle2009}). Indeed, we show that both CBL-representable sets and LC-representable sets are precisely finite unions of polyhedra. Our proof method works through a connection to superlevel sets of piecewise linear convex functions (what we term \emph{polyhedral reverse-convex sets}) that alternately characterize finite unions of polyhedra. In other words, an arbitrary finite union of polyhedra can be modeled as a continuous bilevel program, a linear complementarity problem, or an optimization problem over a polyhedral reverse-convex set.

A natural question arises: how can one relate CBL-representability and MILP-representability? Despite some connections between CBL programs and MILPs (see, for instance, \citet{audet1997links}), the collection of sets they represent are incomparable (see \cref{thm:CBPnotMIP} below). The Jeroslow-Lowe characterization of MILP-representability as the finite union of polytopes summed with a finitely-generated monoid has a fundamentally different geometry than CBL-representability as a finite union of polyhedra. It is thus natural to conjecture that MIBL-representability should involve some combination of the two geometries. We will see that this intuition is \emph{roughly} correct, with an important caveat.


A distressing fact about MIBL programs, noticed early on in \citet{moore1990mixed}, is that the feasible region of a MIBL program may not be \emph{topologically closed} (maybe the simplest example illustrating this fact is Example~1.1 of \citet{Koppe2010}). This throws a wrench in the classical narrative of representability that has largely focused on closed sets. Indeed, the recent work of \citet{lubin2017mixed} is careful to study representability by \emph{closed} convex sets. This focus is entirely justified. Closed sets are indeed of most interest to the working optimizer and modeler, since sets that are not closed may fail to have desirable optimality properties (such as nonexistence of optimal solutions). Accordingly, we aim our investigation on {\em closures} of MIBL-representable sets. In fact, we provide a complete characterization of these sets as unions of finitely many MILP-representable sets (\cref{thm:RepMIBL}). This is our second key result on MIBL-representability. The result conforms to the rough intuition of the last paragraph. MIBL-representable sets are indeed finite unions of other objects, but instead of these objects being polyhedra as in the case of CBL-programs, we now take unions of MILP-representable sets, reflecting the inherent integrality of MIBL programs.

To prove this second key result on MIBL-representability we develop a generalization of Jeroslow and Lowe's theory to mixed integer sets in \emph{generalized polyhedra}, which are finite intersections of closed {\em and} open halfspaces. Indeed, it is the non-closed nature of generalized polyhedra that allows us to study the non-closed feasible regions of MIBL-programs. Specifically, these tools arise when we take the value function approach to bilevel programming, as previously studied in \citep{Lozano2017,dempe2017solving,ye2010new,schaefer2018solving}. Here, we leverage the characterization of \citet{blair1995closed} of the value function of the mixed-integer program in the lower level problem \eqref{eq:intro-lower-level}. Blair's characterization leads us to analyze superlevel and sublevel sets of \emph{Chv\'atal functions}. A Chv\'atal function is (roughly speaking) a linear function with integer rounding (a more formal definition later). \citet{Basu2016} show that superlevel sets of Chv\'atal functions are MILP-representable. Sublevel sets are trickier, but for a familiar reason --- they are, in general, not closed. This is not an accident. The non-closed nature of mixed-integer bilevel sets, generalized polyhedra, and sublevel sets of Chv\'atal functions are all tied together in a key technical result that shows that sublevel sets of Chv\'atal functions are precisely finite unions of generalized mixed-integer linear representable (GMILP-representable) sets (\cref{thm:ChvGomJer}). This result is the key to establishing our second main result on MIBL-representability.

In fact, showing that the sublevel set of a Chv\'atal function is the finite union of GMILP-representable sets is a corollary of a more general result. Namely, we show that the collection of sets that are finite unions of GMILP-representable sets forms an algebra (closed under unions, intersections, and complements). We believe this result is of independent interest. 

The representability results in the mixed-integer case require rationality assumptions on the data. This is an inevitable consequence when dealing with mixed-integer sets. For example, even the classical result of \citet{Jeroslow1984} requires rationality assumptions on the data. Without this assumption the result does not even hold. The key issue is that the convex hull of mixed integer sets are not necessarily polyhedral unless certain other assumptions are made, amongst which the rationality assumption is most common (see \citet{dey2013some} for a discussion of these issues). 

Of course, it is natural to ask if this understanding of MIBL-representability has implications for questions of computational complexity. The interplay between representability and complexity is subtle. We show (in \cref{thm:RepMIBL}) that allowing integer variables in the leader's decision $x$ captures everything in terms of representability as when allowing integer variables in both the leader and follower's decision (up to taking closures). However, we show (in \cref{thm:UpIntisNP}) that the former is in $\mathcal{NP}$ while the latter is $\Sigma_p^2$-complete \citep{Lodi2014}. This underscores a crucial difference between an integer variable in the upper level versus an integer variable in the lower level, from a computational complexity standpoint.

In summary, we make the following contributions. We provide geometric characterizations of CBL-representability and MIBL-representability (where the latter is up to closures) in terms of finite unions of polyhedra and finite unions of MILP-representable sets, respectively. 
In the process of establishing these main results, we also develop a theory of representability of mixed-integer sets in generalized polyhedra and show that finite unions of GMILP-representable sets form an algebra. This last result has the implication that finite unions of MILP-representable sets also form an algebra, up to closures.

The rest of the paper is organized as follows. The main definitions needed to state our main results are found in \cref{s:notation}, followed in \cref{s:key-results} by self-contained statements of these main results. \cref{sub:proof_of_theorem_thm:repequiv} contains our analysis of continuous bilevel sets and their representability. \cref{sub:proof_of_thm_repmibl} explores representability in the mixed-integer setting. \cref{sec:conclusion} concludes.

\subsection*{Some notation}\label{sub:some_notation}

The following basic concepts are entirely standard, but their notation less so. We state our notation for clarity and completeness. Let $\R$, $\Q$, and $\Z$ denote the set of real numbers, rational numbers, and integers, respectively. Given a subset $S$ of $\R^n$ for some integer $n$, the interior and closure of $S$ will be denoted by $\inte(S)$ and $\cl(S)$, respectively. The set of all conic combinations of the elements of $S$ is called the cone of $S$ and denoted $\cone(S)$. The set of all conic combinations with integer multipliers is called the integer cone of $S$ and denoted $\inte\cone(S)$. The affine hull $\aff(S)$ of $S$ is the intersection of all affine sets containing $S$. The relative interior $\relint(S)$ of $S$ is the interior of $S$ in the relative topology of $\aff(S)$. Finally, the ball $B(c,r) = \left\{ x \in \R^n : ||x - c || \le r \right\}$ in the closed ball of radius $r$ centered at $c$. We use $\proj_x \left \lbrace (x,y):(x,y)\in S\right \rbrace$ to denote the projection of the set $S$ on to the space of $x$ variables. For any set $A\subseteq\R^n$, the complement will be denoted by $A^c := \{x\in\R^n: x\not\in A\}$. For any convex set $K\subseteq \R^n$, the recession cone of a convex set $K$ will be denoted by $\rec(K)$. 
We say that a cone $C$ is a {\em pointed} cone if $x\in C \implies -x\not\in C$. 
A cone $C$ is said to be a {\em simplicial }cone if the extreme rays of $C$ are linearly independent.

\section{Key definitions}\label{s:notation}

This section provides the definitions needed to understand the statements of our main results collected in \cref{s:key-results}. Concepts that appear only in the proofs of these results are defined later as needed. 

We begin with formal definitions of the types of sets we study in this paper.

\begin{Def}[Mixed-integer bilevel linear set]\label{Def:MIBL}
A set $S\subseteq \R^{n_\ell+n_f}$ is called a {\em mixed-integer bilevel linear} (MIBL) set if there exist $A\in \R^{m_\ell \times n_\ell},\, B\in\R^{m_\ell \times n_f},\, b\in\R^{m_\ell},\,f\in \R^{n_f},\,D\in\R^{m_f\times n_f},\,C\in\R^{m_f\times n_\ell}$ and $g\in \R^{m_f}$ such that  
	\begin{align}
	S \quad&=\quad S^1\cap S^2\cap S^3, \notag\\
	S^1\quad&=\quad \left\{ (x,\,y)\in\R^{n_\ell+n_f} :\, Ax+By\leq b  \right\}, \notag \\
	S^2\quad&=\quad \left\{ (x,\,y)\in\R^{n_\ell+n_f} :\, y \in \arg\max_y\left\{ f^\top y: Dy \leq g - Cx,\,y_i \in \Z \text{ for }i \in \mathcal I_F \right\}  \right\}, \text{ and } \label{eq:BLP:Bilevel}		\\
	S^3\quad&=\quad \left\{ (x,\,y)\in\R^{n_\ell+n_f} :\, x_i \in \Z \text{ for }i \in \mathcal I_L\right\}, \notag
\end{align}
where $[k]:= \left\{ 1,\,2,\,3,\,\ldots,\,k \right\},\, \mathcal I_L \subseteq [n_\ell],\, \mathcal I_F \subseteq [n_f] $. Further, we call $S$ a 
\begin{itemize}
\item {\em continuous bilevel linear} (CBL) set if it has a representation with $|\mathcal I_L| = |\mathcal I_F| = 0$;
\item {\em bilevel linear with integer upper level} (BLP-UI) set if it has a representation with $|\mathcal I_F| = 0$. 
\end{itemize}
Such sets will be labeled {\em rational} if all the entries in $A,B,C,D, b, f, g$ are rational.
\end{Def}

\begin{Def}[Linear complementarity sets]
A set $S\subseteq \R^n$ is a {\em linear complementarity (LC) set}, if there exist $M\in \R^{n\times n},\,q\in\R^n$ and $A,b$ of appropriate dimensions such that 
\begin{align*}
	S \quad&=\quad \left\{ x\in \R^n:\,x\geq0,\,Mx+q\geq 0,\,x^\top (Mx+q) =0 ,\, Ax\leq b\right\}.
\end{align*}
Sometimes, we represent this using the alternative notation
\begin{align*}
    0\quad\leq\quad x\quad&\perp\quad Mx+q\quad\geq\quad 0\\
    Ax \quad&\leq\quad b.
\end{align*}
A linear complementarity set will be labeled {\em rational} if all the entries in $A,M, b, q$ are rational.
\end{Def}

	As an example, the set of all $n$-dimensional binary vectors is a linear complementarity set. Indeed, they can be modeled as $0\leq x_i\perp(1-x_i)\geq 0$ for $i\in [n]$. 


\begin{Def}[Polyhedral convex function]
	A function $f:\mathbb{R}^n\mapsto \mathbb{R}$ is a {\em polyhedral convex function with $k$ pieces} if there exist $\alpha^1,\ldots,\alpha^k \in \mathbb{R}^n$ and $\beta_1,\ldots,\beta_k \in \mathbb{R}$ such that 
	\begin{equation*}
		f(x) \quad=\quad \max_{j=1}^k \left\{ \left \langle \alpha^j,\,x \right \rangle - \beta_j \right\}.
	\end{equation*}
A polyhedral convex function will be labeled {\em rational} if all the entries in the affine functions are rational.
\end{Def}
	Note that $f$ is a maximum of finitely many affine functions. Hence $f$ is always a convex function. 


\begin{Def}[Polyhedral reverse-convex set]
A set $S\in \R^n$ is a {\em polyhedral reverse-convex (PRC) set} if there exist $n'\in \mathbb{Z}_+,\, A\in \R^{m\times n},\,b\in \R^m $ and polyhedral convex functions $f_i$ for $i\in [n']$ such that
\begin{align*}
	S\quad&=\quad \left\{ x\in\R^n:\,Ax \leq b,\, f_i(x) \geq 0\text{ for }i\in[n'] \right\}.
\end{align*}
A polyhedral reverse-convex set will be labeled {\em rational} if all the entries in $A,b$ are rational, and the polyhedral convex functions $f_i$ are all rational.
\end{Def}
One of the distinguishing features of MIBL sets is their potential for not being closed. This was discussed in the introduction and a concrete example is provided in \cref{lem:StrictIncl}. To explore the possibility of non-closedness we introduce the notion of generalized polyhedra. 


\begin{Def}[Generalized, regular, and relatively open polyhedra]
    A {\em generalized polyhedron} is a finite intersection of open and closed halfspaces. A bounded generalized polyhedron is called a {\em generalized polytope}. The finite intersection of {\em closed} halfspaces is called a {\em regular polyhedron}. A bounded regular polyhedron is a {\em regular polytope}. 
    A {\em relatively open polyhedron} $P$, is a generalized polyhedron such that $P = \relint(P)$.  If relatively open polyhedron $P$ is bounded, we call it a \emph{relatively open polytope}. 

    Such sets will be labeled {\em rational} if all the defining halfspaces (open or closed) can be given using affine functions with rational data.
\end{Def}

Note that the closure of a generalized or relatively open polyhedron is a regular polyhedron. Also, singletons are, by definition, relatively open polyhedra. 


\begin{Def}[Generalized, regular and relatively open mixed-integer sets]
    A \emph{generalized (respectively, regular and relatively open) mixed-integer set} is the set of mixed-integer points in a generalized (respectively, regular and relatively open) polyhedron. 
    
    Such sets will be labeled {\em rational} if the corresponding generalized polyhedra are rational.
    \end{Def}

Our main focus is to explore how collections of the above objects can be characterized and are related to one another. To facilitate this investigation we employ the following notation and vocabulary. Let $\mathscr T$ be a family of sets. These families will include objects of potentially different dimensions. For instance, the family of polyhedra will include polyhedra in $\R^2$ as well as those in $\R^3$. We will often not make explicit reference to the ambient dimension of a member of the family $\mathscr T$, especially when it is clear from context, unless explicitly needed. {For a family $\mathscr T$, the subfamily of bounded sets in $\mathscr T$ will be denoted by $\bar{\mathscr T}$.} Also, $\cl(\mathscr T)$ is the family of the closures of all sets in $\mathscr T$. When referring to the {\em rational} members of a family (as per definitions above), we will use the notation $\mathscr T(\Q)$.

We are not only interested in the above sets, but also linear transformations of these sets. This notion is captured by the concept of representability.

\begin{Def}[Representability]
Given a family of sets $\mathscr T$, $S$ is called a {\em $\mathscr T$-representable set} or {\em representable by $\mathscr T$} if there exists a $T\in \mathscr T$ and a linear transform $L$ such that $S=L(T)$. The collection of all such $\mathscr T$-representable sets is denoted $\mathscr T_R$. We use the notation $\mathscr T_R(\Q)$ to denote the images of the {\em rational} sets in $\mathscr T$ under {\em rational} linear transforms, i.e., those linear transforms that can be represented using rational matrices.\footnote{We will never need to refer to general linear transforms of rational sets, or rational linear transforms of general sets in a family $\mathscr T$; so we do not introduce any notation for these contingencies.}
\end{Def}

\begin{remark}\label{rem:1} 
    The standard definition of representability in the optimization literature uses projections as opposed to general linear transforms. However, under mild assumption on the family $\mathscr T$, it can be shown that $\mathscr T_R$ is simply the collection of sets that are {\em projections} of sets in $\mathscr T$. Since projections are linear transforms, we certainly get all projections in $\mathscr T_R$. Now consider a set $S \in \mathscr T_R$, i.e., there exists a set $T \in \mathscr T$  and a linear transform $L$ such that $S = L(T)$. Observe that $S = \proj_x\{(x,y) : x = L(y), y \in T\}$. Thus, if $\mathscr T$ is a family that is closed under the addition of affine subspaces (like $x = L(y)$ above), and addition of free variables (like the set $\{(x,y): y \in T\}$), then $\mathscr T_R$ does not contain anything beyond projections of sets in $\mathscr T$. All families considered in this paper are easily verified to satisfy these conditions.
\end{remark}

One can immediately observe that $\mathscr T\subseteq\mathscr T_R$ since the linear transform can be chosen as the identity transform. However, the inclusion may or may not be strict. For example, it is well known that if $\mathscr T$ is the set of all polyhedra, then $\mathscr T_R = \mathscr T$. However, if $\mathscr T$ is the family of all (regular) mixed-integer sets then $\mathscr T\subsetneq \mathscr T_R$. 

When referencing specific families of sets we use the following notation. The family of all linear complementarity sets is $\mathscr T^{LC}$, continuous bilevel sets is $\mathscr T^{CBL}$ and the family of all polyhedral reverse-convex sets is $\mathscr T^{PRC}$. The family of mixed-integer sets is $\MILP$. The family of MIBL sets is $\mathscr T^{MIBL}$. The family of BLP-UI sets is $\mathscr T^{BLP-UI}$. We use $\mathscr P$ to denote the family of {\em finite} unions of polyhedra and ${\DMILP}$ to denote the family of finite unions of sets in $\MILP$. 
We use ${\GMILP}$ to denote the family of generalized mixed-integer sets and ${\GDMILP}$ to denote the family of sets that can be written as finite unions of sets in ${\GMILP}$. 
The family of all integer cones is denoted by $\mathscr T^{IC}$. This notation is summarized in \cref{table:set-names}. 

We make a useful observation at this point.

\begin{lemma}\label{lem:DMIRisFiniteUnion}
The family ${\DMILPR}$ is exactly the family of finite unions of MILP-representable sets, i.e., finite unions of sets in $\MILPR$. Similarly, ${\GDMILPR}$ is exactly the family of finite unions of sets in ${\GMILPR}$. The statements also holds for the rational elements, i.e., when we consider ${\DMILPR}(\Q)$, ${\MILPR}(\Q)$, ${\GDMILPR}(\Q)$ and ${\GMILPR}(\Q)$, respectively.
\end{lemma}

\begin{proof} Consider any $T \in \DMILPR$. By definition, there exist sets $T_1, \ldots, T_k \in \MILP$ and a linear transformation $L$ such that $T = L\left(\bigcup_{i=1}^k T_i\right) = \bigcup_{i=1}^k L(T_i)$ and the result follows from the definition of MILP-representable sets. 

Now consider $T\subseteq \R^n$ such that $T = \bigcup_{i=1}^k T^i$ where $T_i\in\MILPR$ for $i\in[k]$. By definition of $T_i\in\MILPR$, we can write $\R^n\supseteq T_i = \left\{ L_i(z^i): z^i\in T_i' \right\}$ where $T_i'\subseteq \R^{n_i}\times \R^{n_i'}$ is a set of mixed integer points in a polyhedron and $L_i: \R^{n_i}\times \R^{n'_i} \mapsto \R^n$ is a linear transform. In other words, $T_i' = \left\{ z^i = (x^i,\,y^i)\in \R^{n_i}\times\Z^{n_i'}: A^ix^i + B^iy^i \leq b^i \right\}$. Let us define
\begin{equation}\label{eq:define-t-i}
\tilde T_i = \left\{ (x,\,x^1,\,\ldots,\,x^k,\,y^1,\,\ldots,\,y^k) : \forall\, j\in[k],\, x^j\in\R^{n_j},\,y^j\in \R^{n_j'};\, x=L_i(x^i),\, A^ix^i + B^iy^i \leq b^i \right\}.
\end{equation}
Clearly, for all $i\in[k]$, $\tilde T_i \subseteq \R^{n+\sum_{j=1}^k \left( n_j+n_j' \right) }$ is a polyhedron and projecting $\tilde T_i\cap \left( \R^{n+\sum_{j=1}^k  n_j }\times\Z^{\sum_{j=1}^k  n_j'} \right) $ over the first $n$ variables gives $T_i$. Let us denote the projection from $(x, x^1, \ldots, x^k, y^1, \ldots, y^k)$ onto the first $n$ variables by $L$. Since any linear operator commutes with finite unions, we can write, 
  \begin{equation*}
    T\quad=\quad\left\{ L(z) : z = (x, x^1, \ldots, x^k, y^1, \ldots, y^k)\in \left( \bigcup_{i=1}^k\tilde T_i  \right) \bigcap \left( \R^{n'}\times\Z^{n''} \right)  \right\},
  \end{equation*}
    where $n' = n+\sum_{j=1}^k  n_j $ and $n'' = \sum_{j=1}^k  n_j' $, proving the first part of the lemma.
    
    The second part about ${\GDMILPR}$ follows along very similar lines and is not repeated here. The rational version also follows easily by observing that nothing changes in the above proof if the linear transforms and all entries in the data are constrained to be rational.
\end{proof}

\begin{remark} Due to \cref{lem:DMIRisFiniteUnion}, we will interchangeably use the notation ${\DMILPR}$ and the phrase ``finite unions of MILP-representable sets" (similarly, ${\GDMILPR}$ and ``finite unions of sets in ${\GMILPR}$") without further comment in the remainder of this paper.\end{remark}

\begin{table}
\begin{center}
{\renewcommand{\arraystretch}{1.3}
\begin{tabular}{|c|l|}
\hline
Notation & \qquad \qquad \qquad \qquad \qquad Family  \\
\hline
$\bar{\mathscr T}$ & bounded sets from the family $\mathscr T$  \\
\hline
$\cl(\mathscr T)$ & the closures of sets in the family $\mathscr T$  \\
\hline
$\mathscr T(\Q)$ & the sets in family $\mathscr T$ determined with rational data  \\
\hline
$\mathscr T_R$ & $\mathscr T$-representable sets  \\
\hline
$\MILP$ & sets that are mixed-integer points in a polyhedron\\
\hline
${\GMILP}$ & sets that are mixed-integer points in a generalized polyhedron\\
\hline
$\mathscr T^{CBL}$ & continuous bilevel linear (CBL) sets  \\
\hline
$\mathscr T^{LC}$ & linear complementarity (LC) sets \\
\hline
$\mathscr T^{BLP-UI}$ & bilevel linear polyhedral sets with integer upper level (BLP-UI) \\
\hline
$\DMILP$ & sets that can be written as finite unions of sets in $\MILP$ \\
 \hline
$\DMILPR$ & sets that can be written as finite unions of sets in $\MILPR$ (see \cref{lem:DMIRisFiniteUnion}) \\
\hline
$\mathscr T^{MIBL}$ & Mixed-integer bilevel sets\\
\hline
${\GDMILP}$ & sets that can be written as finite unions of sets in ${\GMILP}$ \\
\hline
${\GDMILPR}$ & sets that can be written as finite unions of sets in ${\GMILPR}$ (see \cref{lem:DMIRisFiniteUnion}) \\
 \hline
$\mathscr T^{IC}$ & integer cones \\
\hline
$\mathscr P$ & finite unions of polyhedra \\
\hline
\end{tabular}}
\end{center}
\caption{Families of sets under consideration. }\label{table:set-names}
\end{table}

\vskip 5pt
Finally, we introduce concepts that are used in describing characterizations of these families of sets. The key concept used to articulate the ``integrality'' inherent in many of these families is the following.

\begin{Def}[Monoid]
  A set $C\subseteq\R^n$ is a {\em monoid} if for all $x, y\in C$, $x+y\in C$. A monoid is \emph{finitely generated} if there exist $r^1,\,r^2,\,\ldots,\,r^k\in\R^n$ such that 
  \begin{equation*}
    C\quad=\quad \left\{ x: x = \sum_{i=1}^k \lambda_ir^i\text{ where }\lambda_i \in \Z_+;\, \forall\, i\in[k] \right\}.
  \end{equation*}
We will often denote the right-hand side of the above as $\inte\cone \left\{ r^1,\ldots,r^k \right\}$. 
Further, we say that $C$ is a {\em pointed monoid}, if $\cone(C)$ is a  pointed cone. A finitely generated monoid is called {\em rational} if the generators $r^1, \ldots, r^k$ are all rational vectors.
\end{Def}
In this paper, we are interested in discrete monoids. A set $S$ is discrete if there exists an $\eps>0$ such that for all $x\in S$, $B(x,\,\eps)\cap S = \{x\}$. 
Not all discrete monoids are finitely generated. 
For example, the set $M = \left \lbrace (0,\,0)\right \rbrace\cup \left \lbrace x\in\Z^2: x_1\geq 1,\, x_2 \geq 1\right \rbrace $ is a discrete monoid that is not finitely generated.

The seminal result from \citet{Jeroslow1984}, which we restate in \cref{thm:JeroslowLowe}, shows that a rational MILP-representable set is the Minkowski sum of a finite union of rational polytopes and a rational finitely generated monoid. 
%
%
Finally, we define three families of functions that provide an alternative vocabulary for describing ``integrality''; namely, Chv\'atal functions, Gomory functions and Jeroslow functions. These families derive significance here from their ability to articulate value functions of integer and mixed-integer programs (as seen in \citet{Blair1977,Blair1979,Blair1982,blair1995closed}). 

{\em Chv\'atal functions} are defined recursively by using linear combinations and floor ($\floor{\cdot}$) operators on other Chv\'atal functions, assuming that the set of affine linear functions are Chv\'atal functions. We formalize this using a binary tree construction as below. We adapt the definition from \citet{Basu2016}.\footnote{The definition in~\citet{Basu2016} used $\ceil{\cdot}$ as opposed to $\floor{\cdot}$. We make this change in this paper to be consistent with Jeroslow and Blair's notation. Also, what is referred to as ``order" of the Chv\'atal function's representation is called ``ceiling count" in~\citet{Basu2016}.}
\begin{Def}[Chv\'atal functions~\citet{Basu2016}]\label{Def:ChvFun}
A Chv\'atal function $\psi:\R^n\mapsto\R$ is constructed as follows. We are given a finite binary tree where each node of the tree is either: (i) a leaf node which corresponds to an affine linear function on $\R^n$ with rational coefficients; (ii) has one child with a corresponding edge labeled by either $\floor{\cdot}$ or a non-negative rational number; or (iii) has two children, each with edges labeled by a non-negative rational number. Start at the root node and recursively form functions corresponding to subtrees rooted at its children using the following rules. 
\begin{enumerate}
	\item If the root has no children then it is a leaf node corresponding to an affine linear function with rational coefficients. Then $\psi$ is the affine linear function. 
	\item If the root has a single child, recursively evaluating a function $g$, and the edge to the child is labeled as $\floor{\cdot}$, then $\psi(x) = \floor{g(x)}$. If the edge is labeled by a non-negative number $\alpha$, define $\psi(x) = \alpha g(x)$. 
	\item Finally, if the root has two children, containing functions $g_1,\,g_2$ and edges connecting them labeled with non-negative rationals, $a_1,\,a_2$, then $\psi(x) = a_1 g_1(x) + a_2g_2(x)$.
\end{enumerate}
We call the number of $\floor{\cdot}$ operations in a binary tree used to represent a Chv\'atal function the {\em order} of this binary tree representation of the Chv\'atal function. Note that a given Chv\'atal function may have alternative binary tree representations with different orders.  
\end{Def}

\begin{Def}[Gomory functions]\label{Def:ChvGomFun}
	A Gomory function $G$ is the pointwise minimum of finitely many Chv\'atal functions. That is, 
	\begin{equation*}
		G(x)\quad:=\quad \min_{1=1}^k \psi_i(x),
	\end{equation*}
	where $\psi_i$ for $i\in[k]$ are all Chv\'atal functions.
\end{Def}

Gomory functions are then used to build Jeroslow functions, as defined in \citet{blair1995closed}.

\begin{Def}[Jeroslow function]\label{def:JeroslowFunc} 
    Let $G$ be a Gomory function. For any invertible matrix $E$, and any vector $x$, define $\floor{x}_E:= E\floor{E^{-1}x}$. Let $\mathcal{I}$ be a finite index set and let $\{E_i\}_{i\in \mathcal{I}}$ be a set of $n \times n$ invertible rational matrices indexed by $\mathcal{I}$, and $\{w_i\}_{i\in \mathcal{I}}$ be a set of rational vectors in $\R^n$ index by $\mathcal{I}$. Then $J:\R^n\mapsto\R$ is a {\em Jeroslow function} if 
    \begin{equation*}
        J(x) \quad:=\quad \max_{i\in \mathcal{I}} \left\{ G\left (\floor{x}_{E_i}\right ) + w_i^\top \left(x - \floor{x}_{E_i}\right)\right\},
    \end{equation*}
\end{Def}

\begin{remark}
Note that we have explicitly allowed only rational entires in the data defining  Chv\'atal, Gomory and Jeroslow functions. This is also standard in the literature since the natural setting for these functions and their connection to mixed-integer optimization uses rational data.
\end{remark}
\begin{remark}
	Note that it follows from \cref{Def:ChvFun,Def:ChvGomFun,def:JeroslowFunc}, the family of Chv\'atal functions, Gomory functions and Jeroslow function are all closed under composition with affine functions and addition of affine functions.
\end{remark}

A key result in \citet{blair1995closed} is that the value function of a mixed integer program {with rational data} is a Jeroslow function. This result allows us to express the lower-level optimality condition captured in the bilevel constraint \eqref{eq:BLP:Bilevel}. This is a critical observation for our study of MIBL-representability.

We now have all the vocabulary needed to state our main results.

\section{Main results}\label{s:key-results}

Our main results concern the relationship between the sets defined in \cref{table:set-names} and the novel machinery we develop to establish these relationships. 

First, we explore bilevel sets with only continuous variables. We show that the sets represented by continuous bilevel constraints, linear complementarity constraints and polyhedral reverse convex constraints are all equivalent and equal to the family of finite unions of polyhedra.

\begin{theorem}\label{thm:RepEquiv}
The following holds:
\begin{align*}
     \mathscr T_R^{CBL}\quad=\quad \mathscr T_R^{LC}\quad=\quad \mathscr T_R^{PRC} \quad=\quad \mathscr T^{PRC}  \quad=\quad \mathscr P.		
\end{align*}
\end{theorem}


The next result shows the difference between the equivalent families of sets in \cref{thm:RepEquiv} and the family of MILP-representable sets. The characterization of MILP-representable sets by \citet{Jeroslow1984} (restated below as \cref{thm:JeroslowLowe}) is central to the argument here. Using this characterization, we demonstrate explicit examples of sets that illustrate lack of containment in these families.

\begin{corollary}\label{thm:CBPnotMIP}
The following holds:
		\begin{align*}
		\mathscr T_R^{CBL}\setminus \MILPR \quad&\neq\quad \emptyset \text{ and }\\
		\MILPR\setminus \mathscr T_R^{CBL} \quad&\neq\quad \emptyset.
	\end{align*}
\end{corollary}

The next result shows that the lack of containment of these families of sets arises because of unboundedness. 

\begin{corollary}\label{thm:BoundedAllSame}
The following holds: 
	\begin{align*}
	\bar{\mathscr T_R^{CBL}} \quad=\quad 	\bar{\mathscr T_R^{LC}} \quad=\quad \bar{\mathscr T_R^{PRC}} \quad=\quad \bar{\mathscr{P}} \quad=\quad \bar{\MILPR}. 
	\end{align*}
\end{corollary}

\begin{remark}\label{rem:rational-version} The rational versions of \cref{thm:RepEquiv}, \cref{thm:CBPnotMIP}, \cref{thm:BoundedAllSame} all hold, i.e., one can replace all the sets in the statements by their rational counterparts. For example, the following version of \cref{thm:RepEquiv} holds:
\begin{align*}
    \mathscr T_R^{CBL}(\Q)\quad=\quad \mathscr T_R^{LC}(\Q)\quad=\quad \mathscr T_R^{PRC}(\Q) \quad=\quad \mathscr T^{PRC}(\Q)  \quad=\quad \mathscr P(\Q).
\end{align*}
We will not explicitly prove the rational versions; the proofs below can be adapted to the rational case without any difficulty.
\end{remark}

Our next set of results concern the representability by bilevel problems with integer variables. We show that, with integrality constraints in the upper level only, bilevel representable sets correspond to finite unions MILP-representable sets. Further allowing integer variables in lower level may yield sets that are not necessarily closed. However, we show that the closure of sets are again finite unions of MILP-representable sets. In contrast to Remark~\ref{rem:rational-version}, rationality of the data is an important assumption in this setting. This is to be expected: this assumption is crucial whenever one deals with mixed-integer points, as mentioned in the introduction.
\begin{theorem}\label{thm:RepMIBL}
The following holds:
\begin{equation*}
   \mathscr T_R^{MIBL}(\Q)\quad\supsetneq\quad\cl\left (\mathscr T_R^{MIBL}(\Q)\right )\quad=\quad\mathscr {T}_R^{BLP-UI}(\Q) \quad=\quad {\DMILPR}(\Q).
\end{equation*}
In fact, in the case of BLP-UI sets, the rationality assumption can be dropped; i.e., 
\begin{equation*}
   \mathscr {T}_R^{BLP-UI} \quad=\quad {\DMILPR}.
\end{equation*}
\end{theorem}
Behind the proof of this result are two novel technical results that we believe have interest in their own right. The first concerns an algebra of sets that captures, to some extent, the inherent structure that arises when bilevel constraints and integrality interact. Recall that an algebra of sets is a collection of sets that is closed under taking complements and unions. It is trivial to observe that finite unions of generalized polyhedra form an algebra, i.e., a family that is closed under finite unions, finite intersections and complements. We show that a similar result holds even for finite unions of generalized mixed-integer representable sets.
\begin{theorem}\label{thm:algebra}
    The family of sets $\left \lbrace S \subseteq\R^n:S \in \GDMILPR(\Q) \right \rbrace$ is an algebra over $\R^n$ for any $n$.  
\end{theorem}


The connection of the above algebra to optimization is made explicit in the following theorem, which is used in the proof of \cref{thm:RepMIBL}.
\begin{theorem}\label{thm:ChvGomJer}
    Let $\psi:\R^n\mapsto\R$ be a Chv\'atal, Gomory, or Jeroslow function. Then (i) $\left \lbrace x:\psi(x) \leq 0\right \rbrace$ (ii) $\left \lbrace x:\psi(x) \geq 0\right \rbrace$ (iii) $\left \lbrace x:\psi(x) = 0\right \rbrace$ (iv) $\left \lbrace x:\psi(x) < 0\right \rbrace$ and (v) $\left \lbrace x:\psi(x) > 0\right \rbrace$ are elements of ${\GDMILPR}(\Q)$.
\end{theorem}
As observed below \cref{def:JeroslowFunc}, the family of Jeroslow functions capture the properties of the bilevel constraint \eqref{eq:BLP:Bilevel} and thus \cref{thm:ChvGomJer} proves critical in establishing \cref{thm:RepMIBL}.

Finally, we also discuss the computational complexity of solving bilevel programs and its the connections with representability. A key result in this direction is the following.

\begin{theorem}\label{thm:UpIntisNP}
If $S$ is a rational BLP-UI set, then the sentence ``Is $S$ non-empty?'' is in $\mathcal{NP}$.
\end{theorem}

\begin{remark} In light of the \cref{thm:RepMIBL,thm:UpIntisNP}, we observe the following. While adding integrality constraints in the upper level improves the modeling power, it does not worsen the theoretical difficulty to solve such problems. We compare this with the results of \citep{Lodi2014} which says that if there are integral variables in the lower level as well, the problem is much harder ($\Sigma_p^2$-complete). However, by \cref{thm:RepMIBL}, this does not improve modeling power up to closures. This shows some of the subtle interaction between notions of complexity and representability.
\end{remark}

\section{Representability of continuous bilevel sets}\label{sub:proof_of_theorem_thm:repequiv}

The goal of this section is to prove \cref{thm:RepEquiv}. We establish the result across several lemmata. First, we show that sets representable by polyhedral reverse convex constraints are unions of polyhedra and vice versa. Then we show three inclusions, namely $\mathscr T_R^{PRC} \subseteq \mathscr T_R^{CBL}$, $\mathscr T_R^{CBL} \subseteq \mathscr T_R^{LC}$ and $\mathscr T_R^{LC} \subseteq \mathscr T_R^{PRC}$. The three inclusions finally imply \cref{thm:RepEquiv}.

Now we prove a useful technical lemma. 

 \begin{lemma}\label{thm:ReprIncl}
If $\mathscr T^1$ and $\mathscr T^2$ are two families of sets such that $\mathscr T^1 \subseteq \mathscr T^2_R$ then $\mathscr T^1_R \subseteq \mathscr T^2_R$. Moreover, the rational version holds, i.e., $\mathscr T^1(\Q) \subseteq \mathscr T^2_R(\Q)$ implies $\mathscr T^1_R(\Q) \subseteq \mathscr T^2_R(\Q)$.
\end{lemma}
\begin{proof}
  Let $T\in \mathscr T^1_R$. This means there is a linear transform $L^1$ and $T^1\in\mathscr T^1$ such that $T = L^1(T^1)$. Also, this means $T^1 \in \mathscr T^2_R$, by assumption. So there exists a linear transform $L^2$ and $T^2\in \mathscr T^2$ such that $T^1 = L^2(T^2)$. So $T = L^1(T^1) = L^1(L^2(T^2)) = (L^1\circ L^2) (T^2)$, proving the result. The rational version follows by restricting all linear transforms and sets to be rational.
\end{proof}

We now establish the first building block of \cref{thm:RepEquiv}.

\begin{lemma}\label{lem:PRCisPolyhedunion}
The following holds: 
\begin{equation*}
\mathscr P \quad=\quad \mathscr T^{PRC} \quad=\quad \mathscr T_R^{PRC}.
\end{equation*}
\end{lemma}

\begin{proof}
	We start by proving the first equivalence. Consider the $\supseteq$ direction first.
  Let $S\in \mathscr T^{PRC}$. Then $S = \left\{ x\in\R^n:\,Ax \leq b,\, f_i(x) \geq 0\text{ for }i\in[n'] \right\}$ for some polyhedral convex functions $f_i$ with $k_i$ pieces each. First, we show that $S$ is a finite union of polyhedra. Choose one  halfspace from the definition of each of the functions $f_i(x) = \max_{j=1}^{k_i} \{\langle \alpha^{ij} \rangle - \beta_j\}$ (i.e., $\left\{ x: \langle \alpha^{ij},\,x \rangle - \beta_{ij} \geq 0 \right\}$ for some $j$ and each $i$) and consider their intersection. This gives a polyhedron. There are exactly $K=\prod_{i=1}^{n'}k_i$ such polyhedra. We claim that $S$ is precisely the union of these $K$ polyhedra, intersected with $\left\{ x:Ax\leq b \right\}$ (clearly, the latter set is in $\mathscr P$). Suppose $x\in S$. Then $Ax\leq b$. Also since $f_i(x) \geq 0$ , we have $\max_{j=1}^{k_i } \left\{ \langle \alpha^{ij},\,x \rangle - \beta_{ij}  \right\} \geq 0$. This means for each $i$, there exists a $j_i$ such that $\langle \alpha^{ij_i},\,x \rangle - \beta_{ij_i} \geq 0$. The intersection of all such halfspaces is one of the $K$ polyhedra defined earlier. Conversely, suppose $x$ is in one of these $K$ polyhedra (the one defined by $\langle \alpha^{ij_i},\,x \rangle - \beta_{ij_i} \geq 0$ for $i\in [n']$) intersected with $\left\{ x:Ax\leq b \right\}$. Then, $f_i(x) = \max_{j=1}^{k_i} \left\{ \langle \alpha^{ij},\,x \rangle - \beta^{ij} \right\} \geq \langle \alpha^{ij_i},\,x \rangle - \beta_{ij_i} \geq 0$ and thus $x \in S$. This shows that $\mathscr T^{PRC} \subseteq \mathscr P$.

	Conversely, suppose $P\in \mathscr P$ and is given by $P = \bigcup_{i=1}^{k}P_i$ and $P_i = \left\{ x: A^ix\leq b^i \right\}$ where $b^i \in \R^{m_i}$. Let $a^i_j$ refer to the $j$-th row of $A^i$ and $b^i_j$ the $j$-th coordinate of $b^i$. Let $\Omega$ be the Cartesian product of the index sets of constraints, i.e., $\Omega = \{1, \ldots, m_1\} \times \{1, \ldots, m_2\} \times \ldots \times \{1, \ldots, m_k\}$. For any $\omega \in \Omega$, define	
	Now consider the following function:
	\begin{equation*}
		f_\omega(x) \quad=\quad \max_{i=1}^k \left\{ -\langle a^i_{\omega_i}, \,x \rangle + b^i_{\omega_i} \right\},
	\end{equation*}
    where $\omega_i$ denotes the index chosen in $\omega$ from the set of constraints in $A^i x \leq b^i$, $i=1, \ldots, k$. This construction is illustrated in \cref{fig:PRC}. Let $T = \left\{ x:\, f_\omega (x) \geq 0,\forall\,\omega\in\Omega \right\} \in \mathscr T^{PRC}$. We now claim that $P=T$. If $\bar x\in P$ then there exists an $i$ such that $\bar x\in P_i$, which in turn implies that $b^i - A^i\bar x\geq 0$. However each of the $f_\omega$ contains at least one of the rows from $b^i-A^i\bar x$, and that is non-negative. This means each $f_\omega$, which are at least as large as any of these rows, are non-negative. This implies $P\subseteq T$. Now suppose $\bar x \not\in P$. This means in each of the $k$ polyhedra $P_i$, at least one constraint is violated. Now consider the $f_\omega$ created by using each of these violated constraints. Clearly for this choice, $f_\omega(x) < 0$. Thus $\bar x\not\in T$. This shows $T\subseteq P$ and hence $P=T$. This finally shows $\mathscr P\subseteq \mathscr T^{PRC}$. Combined with the argument in the first part of the proof, we have $\mathscr T^{PRC} = \mathscr P$.
	\begin{figure}
	\centering
	\includegraphics{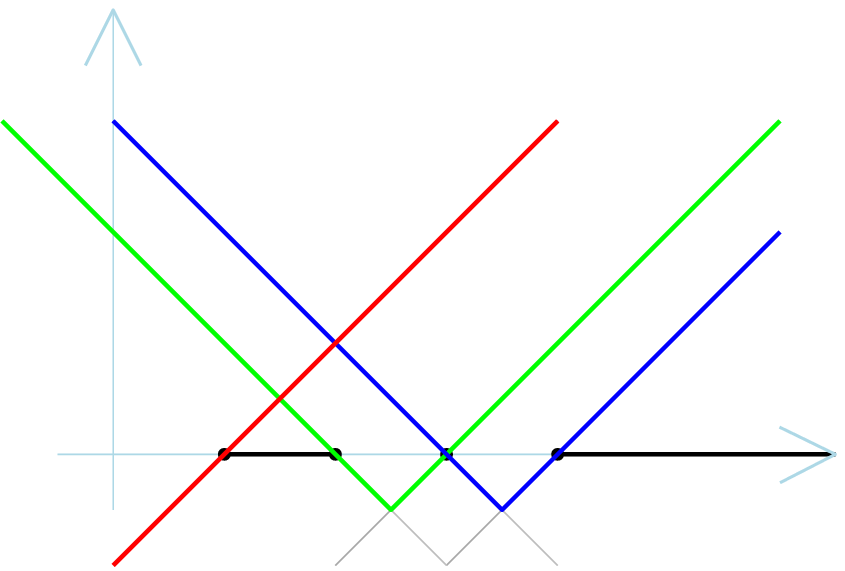}
        \caption{Representation of union of polyhedra $[1,\,2]\cup \{3\}\cup [4,\,\infty)$ as a PRC set. The three colored lines correspond to three different the polyhedral convex functions. 
        The points where each of those functions
    are non-negative precisely correspond to the region shown in black, the set which we wanted to represent in the first place.}\label{fig:PRC}
	\end{figure}

	Now consider the set $\mathscr P_R$. A linear transform of a union of finitely many polyhedra {\em is} a union of finitely many polyhedra. Thus $\mathscr P_R = \mathscr P$. But $\mathscr P = \mathscr T^{PRC}$ and so $\mathscr T_R^{PRC} = \mathscr P_R = \mathscr P = \mathscr T^{PRC}$, proving the remaining equivalence in the statement of the lemma. 
\end{proof}

Next, we show that any set representable using polyhedral reverse convex constraints is representable using continuous bilevel constraints. To achieve this, we give an explicit construction of the bilevel set.

\begin{lemma}\label{lem:PRCinCBP}
The following holds: 
\begin{equation*}
\mathscr T_R^{PRC} \quad\subseteq\quad \mathscr T_R^{CBL}.
\end{equation*}
\end{lemma}
\begin{proof}
Suppose $S\in \mathscr T^{PRC}$. Then, 
\begin{equation*}
	S \quad=\quad \left\{ x\in\R^n:\,Ax \leq b,\, f_i(x) \geq 0\text{ for }i\in[n'] \right\}
\end{equation*}
 for some $n',\,A,\,b$ and polyhedral convex functions $f_i:\R^n\mapsto\R$ for $i\in[n']$. Further, let us explicitly write $f_i(x) = \max_{j=1}^{k_i} \left\{ \langle \alpha^{ij} ,x \rangle - \beta_{ij} \right\}$ for $j\in [k_i]$ for $i \in [n']$. Note for any $i$, $f_i(x) \geq 0$ if and only if $ \left\{ \langle \alpha^{ij},\,x \rangle - \beta_{ij} \right\} \geq 0  $ for each $j \in [k_i]$. Now, consider the following CBL set $S' \subseteq \R^n\times\R^{n'}$ where $(x,y) \in S'$ if
	\begin{align}
	Ax  \quad&\leq\quad b \notag \\
	y\quad&\geq\quad 0 \notag \\
	y\quad&\in\quad \arg\max_{y} \left\{ -\sum_{i=1}^{n'}y_i:\, y_i \geq \langle \alpha^{ij} ,x \rangle - \beta_{ij} \text{ for } j\in[k_i],\,i\in[n'] \right\}.\label{eq:t1}
\end{align}
    A key observation here is that in \eqref{eq:t1}, the condition $y_i \geq \langle \alpha^{ij} ,x \rangle - \beta_{ij}$ for all $j\in[k_i]$ is equivalent to saying $y_i \geq f_i(x)$. Thus \eqref{eq:t1} can be written as $	y\in \arg\min_{y} \left\{ \sum_{i=1}^{n'}y_i:\, y_i \geq f_i(x),\,i\in[n'] \right\}$. Since we are minimizing the sum of coordinates of $y$, this is equivalent to saying $y_i = f_i(x)$ for $i\in[n']$. However, we have constrained $y$ to be non-negative. So, if $S'' = \left\{ x\in\R^n:\,\exists\,y\in\R^{n'} \text{ such that } (x,\,y)\in S'\right\}$, it naturally follows  that $S=S''$. Thus $S\in \mathscr T^{CBL}_R$, proving the inclusion $\mathscr T^{PRC} \subseteq \mathscr T_R^{CBL}$. The result then follows from \cref{thm:ReprIncl}.
 \end{proof}

The next result shows that a given CBL-representable set can be represented as a LC-representable set. Again, we give an explicit construction.

\begin{lemma}\label{lem:CBPinLC}
The following holds:	
\begin{equation*}
\mathscr T_R^{CBL} \quad\subseteq\quad \mathscr T_R^{LC}.
\end{equation*}
\end{lemma}

\begin{proof}
	Suppose $S\in \mathscr T^{CBL}$. Let us assume that the parameters $A,\,B,\,b,\,f,\,C,\,D,\,g$ that define $S$ according to \cref{Def:MIBL} has been identified. We first show that $S\in \mathscr T_R^{LC}$. Retaining the notation in \cref{Def:MIBL}, let $(x,\,y) \in S$. Then from \eqref{eq:BLP:Bilevel}, $y$ solves a linear program that is parameterized by $x$. The strong duality conditions (or the KKT conditions) can be written for this linear program are written below as \eqref{eq:BiKKTstart} - \eqref{eq:BiKKTend} and hence $S$ can be defined as $(x,\,y)$ satisfying the following constraints for some $\lambda$ of appropriate dimension:
	\begin{subequations}
	\begin{align}
	Ax+By \quad&\leq\quad b\\
	D^\top \lambda -f \quad&=\quad 0\label{eq:BiKKTstart}\\
	g -Cx - Dy \quad&\geq\quad 0\\
	\lambda\quad&\geq\quad 0\\
	(g-Cx-Dy)^\top \lambda \quad&=\quad 0. \label{eq:BiKKTend}
\end{align}\label{eq:BiKKT}
\end{subequations}
Consider $S'$ as the set of $(x,\, y,\,\lambda)$  satisfying
	\begin{alignat*}{30}
		\quad&\quad  \quad&&\quad b-Ax-By \quad&\geq&\quad 0\\
		\quad&\quad  \quad&&\quad D^\top \lambda - f \quad&\geq&\quad 0\\
		0\quad&\leq\quad \lambda \quad&\perp&\quad g-Cx-Dy \quad&\geq&\quad 0.
	\end{alignat*}

Clearly, $S'\in \mathscr T_R^{LC}$. Let $S'' = \left\{ (x,\,y):  (x,\, y,\,\lambda) \in S' \right\}$. Clearly $S''\in\mathscr T_R^{LC}$. We now argue that $S=S''$. This follows from the fact that Lagrange multipliers $\lambda$ exist for the linear program in \eqref{eq:BLP:Bilevel} so that $(x,\,y,\,\lambda)\in S'$. 
Hence $\mathscr T^{CBL} \subseteq \mathscr T_R^{LC}$ follows and by \cref{thm:ReprIncl}, the result follows. %
\end{proof}

Finally, to complete the cycle of containment, we show that a set representable using linear complementarity constraint can be represented as a polyhedral reverse convex set.

\begin{lemma}\label{lem:LCinPRC}
The following holds:	
\begin{equation*}
\mathscr T_R^{LC} \quad\subseteq\quad \mathscr T_R^{PRC}.
\end{equation*}
\end{lemma}
\begin{proof}
	{Let $p,\,q\in\R$. Notice that $0\leq p \perp q \geq 0 \iff p\geq 0,\,q\geq 0,\, \max\{-p,\,-q\} \geq 0$. Consider any $S \in \mathscr T^{LC}$. Then $$S = \left\{ x:\, 0\leq x \perp Mx+q \geq 0,\,Ax\leq b \right\} = \left\{ x:\,Ax\leq b,\, x\geq0,\,Mx+q\geq 0,\, f_i(x) \geq 0 \right\},$$ where $f_i(x) = \max \{-x_i,\, -[Mx+q]_i\}$. Clearly, each $f_i$ is a polyhedral convex function and hence by definition $S$ is a polyhedral reverse-convex set. This implies $S\in \mathscr T^{PRC}$.  So, $\mathscr T^{LC}\subseteq \mathscr T^{PRC} \implies \mathscr T_R^{LC} \subseteq \mathscr T_R^{PRC}$.}
\end{proof}


With the previous lemmata in hand we can now establish \cref{thm:RepEquiv}. 

\begin{proof}[Proof of \cref{thm:RepEquiv}]
Follows from \cref{lem:PRCisPolyhedunion,lem:PRCinCBP,lem:CBPinLC,lem:LCinPRC}.
\end{proof}

We turn our focus to establishing \cref{thm:CBPnotMIP}. This uses the following seminal result of \citet{Jeroslow1984}, which gives a geometric characterization of MILP-representable sets.
\begin{theorem}[\citet{Jeroslow1984}]\label{thm:JeroslowLowe}
	Let $T\in \MILPR(\Q)$. Then,
	\begin{equation*}
		T\quad=\quad P+C
	\end{equation*}
	for some $P \in \bar{\mathscr P}(\Q),\, C \in \mathscr T^{IC}(\Q)$.
\end{theorem}

We now give two concrete examples to establish \cref{thm:CBPnotMIP}. First, a set which is CBL-representable but not MILP-representable and, second, an example of a set which is MILP-representable but not CBL-representable.
\begin{proof}[Proof of \cref{thm:CBPnotMIP}]
\begin{figure}
\centering
	\includegraphics[width=0.5\textwidth]{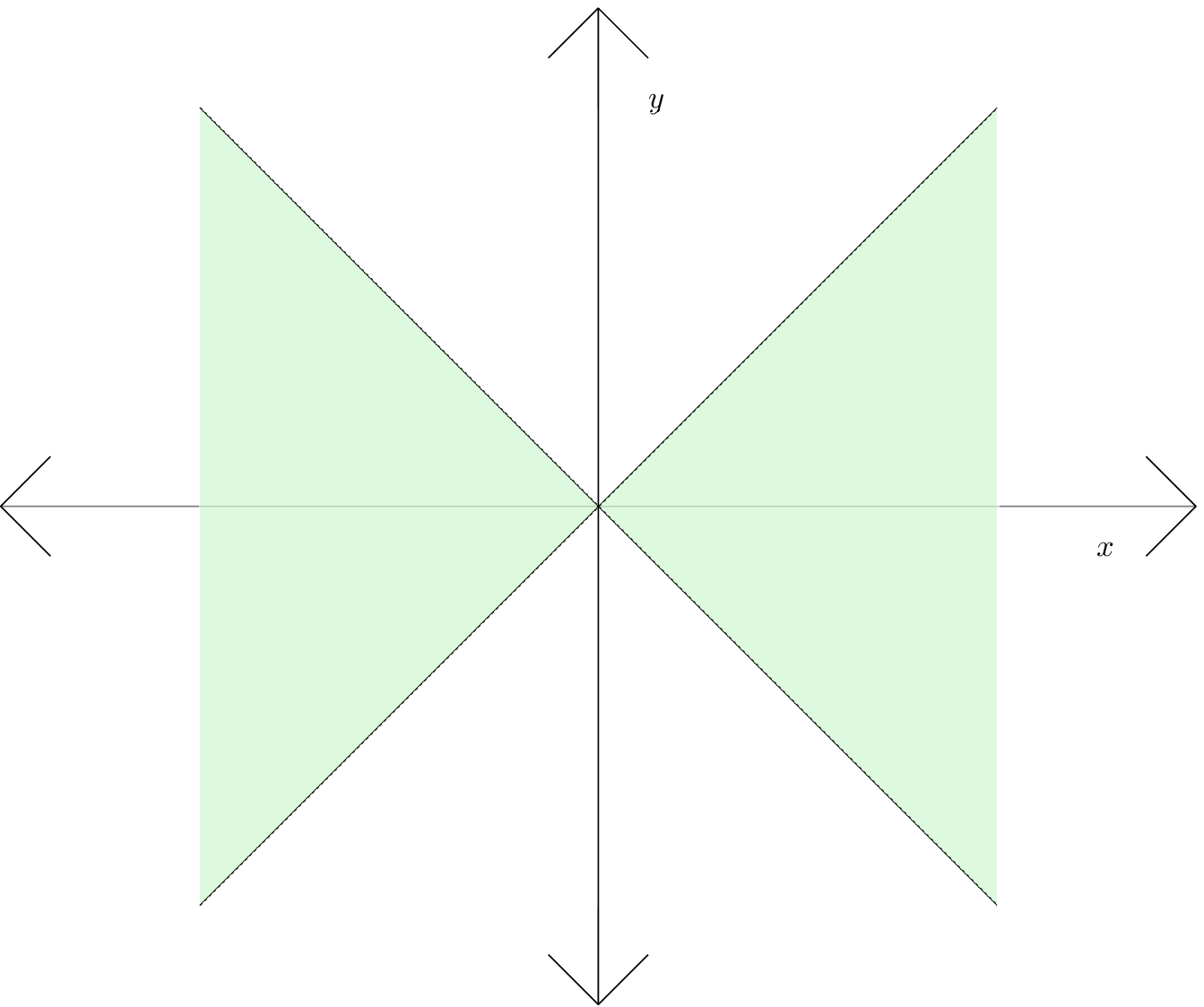}
	\caption{The set $T$ used in the proof of \cref{thm:CBPnotMIP}. 
  Note that $T\in\mathscr T_R^{CBL}\setminus T_R^{MI}$
  }
	\label{fig:Bilevelrepr}
\end{figure}
We construct a set $T \in \mathscr{T}^{CBL}_R$ as follows. Consider the following set $T'\in\mathscr T^{CBL}$ given by $(x,\,y,\,z_1,\,z_2) \in \R^4$ satisfying:
\begin{align*}
 	(y,z_1,z_2) \quad & \in\quad \arg\min \left\{ z_1 - z_2 : 
   \begin{array}{rl}	
 		z_1  &\geq x\\
 		z_1  &\geq -x\\
 		z_2 &\leq x\\
 		z_2 &\leq -x\\
 		y &\leq z_1\\
 		y &\geq z_2
 	\end{array}
   \right\}
\end{align*}
with no upper-level constraints. Consider $T = \left\{ (x,\,y)\in\R^2: (x,\,y,\,z_1,\,z_2)\in T' \right\}$ illustrated in \cref{fig:Bilevelrepr}. Note that $T \in \mathscr{T}^{CBL}_R$. We claim $T \not\in \MILPR$. Suppose it is. Then by \cref{thm:JeroslowLowe}, $T$ is the Minkowski sum of a finite union of polytopes and a monoid. Note that $\left\{ (x,\,x): x\in\R \right\}\subset T$ which implies $(1,\,1)$ is an extreme ray and $\lambda(1,\,1)$ should be in the integer cone of $T$ for some $\lambda > 0$. Similarly $\left\{ (-x,\,x):x\in\R \right\}\subset T$ which implies $(-1,\,1)$ is an extreme ray and $\lambda'(-1,\,1)$ should be in the integer cone of $T$ for some $\lambda' > 0$. Both the facts imply, for some $\lambda'' > 0$, the point $(0,\,\lambda'') \in T$. But no such point is in $T$ showing that $T\not \in \MILPR$. 

Conversely, consider the set of integers $\Z \subseteq \R$. Clearly $\Z\in \MILPR$ since it is the Minkowski sum of the singleton polytope $\{0\}$ and the integer cone generated by $-1$ and $1$. Suppose, by way of contradiction, that $\Z$ can be expressed as a finite union of polyhedra (and thus in $\mathscr T^{CBL}$ by \cref{thm:RepEquiv}). Then there must exist a polyhedron that contains infinitely many integer points. Such a polyhedron must be unbounded and hence have a non-empty recession cone. However, any such polyhedron has non-integer points. This contradicts the assumption that $\Z$ is a finite union of polyhedra.  
\end{proof}

These examples show that the issue of comparing MILP-representable sets and CBL-representable sets arises in how these two types of sets can become unbounded. MILP-representable sets are unbounded in ``integer'' directions from a single integer cone, while CBL-representable sets are unbounded in ``continuous'' directions from potentially a number of distinct recession cones of polyhedra. Restricting to bounded sets removes this difference.

\begin{proof}[Proof of \cref{thm:BoundedAllSame}]
	The first three equalities follow trivially from \cref{thm:RepEquiv}. To prove that $\bar{\mathscr{P}} = \bar{\MILPR}$, observe from \cref{thm:JeroslowLowe} that any set in $\MILPR$ is the Minkowski sum of a finite union of polytopes and a monoid. Observing that $T\in\MILPR$ is bounded if and only if the monoid is a singleton set containing only the zero vector, the equality follows. 
\end{proof}

\section{Representability of mixed-integer bilevel sets} 
\label{sub:proof_of_thm_repmibl}

The goal of this section is to prove \cref{thm:RepMIBL}. Again, we establish the result over a series of lemmata. In \cref{sub:mixed_integer_bilevel_sets_with_continuous_lower_level}, we show that the family of sets that are representable by bilevel linear polyhedral sets with integer upper-level variables is equal to the family of finite unions of MILP-representable sets (that is, $\mathscr {T}_R^{BLP-UI} = {\DMILPR}$). In \cref{sub:technical_lemmata} we establish an important intermediate result to proving the rest of \cref{thm:RepMIBL}, that the family ${\GDMILPR}$ is an algebra of sets; that is, it is closed under unions, intersections, and complements (\cref{thm:algebra}). This pays off in \cref{sub:general_mixed_integer_bilevel_sets}, where we establish the remaining containments of \cref{thm:RepMIBL} using the properties of this algebra and the characterization of value functions of mixed-integer programs in terms of Jeroslow functions, due to \citet{blair1995closed}. 

\subsection{Mixed-integer bilevel sets with continuous lower level} 
\label{sub:mixed_integer_bilevel_sets_with_continuous_lower_level}

First, we show that any BLP-UI set is the finite union of MILP-representable sets.

\begin{lemma}\label{lem:DMIPinUIMIBL}
The following holds:
\begin{align*}
\mathscr T_R^{BLP-UI} \quad&\subseteq\quad \DMILPR
\end{align*}
Moreover, the same inclusion holds in the rational case; i.e. $\mathscr T_R^{BLP-UI}(\Q) \subseteq \DMILPR(\Q)$.
\end{lemma}
\begin{proof}
    Suppose $T\in \mathscr T^{BLP-UI}$. Then $T = \left\{ x: x\in T' \right\} \cap \left\{ x : x_j \in \Z,\,\forall\, j\in \mathcal I_L \right\}$, for some $T' \in \mathscr T^{CBL}$ and for some $\mathcal{I}_L$ (we are abusing notation here slightly because we use the symbol $x$ now to denote both the leader and follower variables in the bilevel problem).    
    By \cref{thm:RepEquiv}, $T' \in \mathscr P$. Thus we can write $T' = \bigcup_{i=1}^k T'_i$ where each $T'_i$ is a polyhedron. Now $T = \left( \bigcup_{i=1}^k T'_i \right) \cap \left\{ x : x_j \in \Z,\,\forall\, j\in \mathcal I_L \right\} = \bigcup_{i=1}^k \left( \left\{ x : x_j \in \Z,\,\forall\, j\in \mathcal I_L \right\} \cap T'_i \right) $ which is in $\DMILP$ by definition. 
    The result then follows from \cref{thm:ReprIncl}. The rational version involving $\mathscr T_R^{BLP-UI}(\Q)$ and $\DMILPR(\Q)$ holds by identical arguments restricting to rational data.
 \end{proof}

We are now ready to prove the reverse containment to \cref{lem:DMIPinUIMIBL} and thus establish one of the equivalences in \cref{thm:RepMIBL}.

\begin{lemma}\label{lem:UIMIBLinDMIP}
The following holds:	
\begin{equation*}
\DMILPR \quad\subseteq\quad \mathscr T_R^{BLP-UI}.
\end{equation*}
Moreover, the same inclusion holds in the rational case; i.e. $\mathscr T_R^{BLP-UI}(\Q) \subseteq \DMILPR(\Q)$.
\end{lemma}
\begin{proof}
    Let $T\in \DMILPR$; then, by definition, there exist polyhedra $\tilde T_i$ such that $T$ is the linear image under the linear transform $L$ of the mixed integer points in the union of the $\tilde T_i$.  Let $\tilde T = \bigcup_{i=1}^k\tilde T_i$ and so $T = \left\{ L(y,z): (y,z) \in\tilde T\cap \left( \R^{n'}\times\Z^{n''} \right)  \right\}$. By \cref{thm:RepEquiv}, $\tilde T\in \mathscr T_R^{CBL}$.  If the $z \in \Z^{n''}$ are all upper level variables in the feasible region of the CBL that projects to $\tilde T$, then $T$ is clearly in $\mathscr T_R^{BLP-UI}$. If, on the other hand, some $z_i$ is a lower level variable then adding a new integer upper level variable $w_i$ and upper level constraint $w_i = z_i$ gives a BLP-UI formulation of a lifting of $\tilde T$, and so again $T \in \mathscr T_R^{BLP-UI}$. This proves the inclusion. The rational version involving $\mathscr T_R^{BLP-UI}(\Q)$ and $\DMILPR(\Q)$ holds by identical arguments restricting to rational data.
%
\end{proof}


\subsection{The algebra ${\GDMILPR}$ } 
\label{sub:technical_lemmata}


We develop some additional theory for generalized polyhedra, as defined in \cref{s:notation}. The main result in this section is that the family of sets that are finite unions of sets representable by mixed-integer points in generalized polyhedra forms an algebra (in the sense of set theory). Along the way, we state some standard results from lattice theory and prove some key lemmata leading to the result. This is used in our subsequent proof in the representability of mixed-integer bilevel problems.

We first give a description of an arbitrary generalized polyhedron in terms of relatively open polyhedra. This allows us to extend properties we prove for relatively open polyhedra to generalized polyhedra.
\begin{lemma}\label{lem:genpolyIsUnionRelOpPoly}
Every generalized polyhedron is a finite union of relatively open polyhedra. If the generalized polyhedron is rational, then the relatively open polyhedra in the union can also be taken to be rational.
\end{lemma}
\begin{proof}
We proceed by induction on the affine dimension $d$ of the polyhedron. The result is true by definition for $d=0$, where the only generalized polyhedron is a singleton, which is also a relatively open polyhedron. For higher dimensions, let $P = \left \lbrace x:Ax < b,\,Cx \leq d \right \rbrace$ be a generalized polyhedron.  Without loss of generality, assume $P$ is full-dimensional, because otherwise we can work in the affine hull of $P$. We then write
    \begin{align*}
        P \quad=\quad \left \lbrace x : \begin{array}{c}Ax < b\\Cx<d \end{array} \right \rbrace \cup \bigcup_{i} \left ( P\cap \left \lbrace x: \langle c^i,\,x\rangle = d_i\right \rbrace\right ).
    \end{align*}
    The first set is a relatively open polyhedron and each of the sets in the second union is a generalized polyhedron of lower affine dimension, each of which are finite unions of relatively open polyhedra by the inductive hypothesis.
    
    The rational version also goes along the same lines, by restricting the data to be rational.
\end{proof}

From~\cite[Theorem 6.6]{Rockafellar1970}, we obtain the following:

\begin{lemma}[\citet{Rockafellar1970}]\label{lem:ProjIsRelOpen}
    Let $Q\subseteq\R^n$ be a relatively open polyhedron and $L$ is any linear transformation. Then $L(Q)$ is relatively open. If $Q$ and $L$ are both rational, then $L(Q)$ is also rational.
\end{lemma}

We now prove for relatively open polyhedra an equivalent result to the Minkowski-Weyl theorem for regular polyhedra.

\begin{lemma}\label{lem:MinkWeyl}
    Let $Q \subseteq \R^n$ be a relatively open polyhedron. Then $Q = P+R$ where $P$ is a relatively open {\em polytope} and $R$ is the recession cone $\rec\cl(Q)$. If $Q$ is rational then $P$ and $R$ can also be taken to be rational.
\end{lemma}
\begin{proof} We first assume that $\dim(Q) = n$; otherwise, we can work in the affine hull of $Q$. Thus, we may assume that $Q$ can be expressed as $Q = \{x \in \R^n: Ax < b\}$ for some matrix $A$ and right hand side $b$. 

    We first observe that if $L$ is the lineality space of $\cl(Q)$, then $Q = \proj_{L^\perp}(Q) + L$ where $\proj_{L^\perp}(\cdot)$ denotes the projection onto the orthogonal subspace $L^\perp$ to $L$. Moreover, $\proj_{L^\perp}(Q)$ is a relatively open polyhedron by \cref{lem:ProjIsRelOpen} and its closure is pointed (since we projected out the lineality space). Therefore, it suffices to establish the result for full-dimensional relatively open polyhedra whose closure is a pointed (regular) polyhedron. Henceforth, we will assume $\cl(Q)$ is a pointed polyhedron.

Define $Q_\epsilon := \{x \in \R^n: Ax \leq b - \epsilon\mathbf 1\}$ and observe that $Q_0 = \cl(Q)$ and $Q = \cup_{0 < \epsilon \leq 1} Q_\epsilon$. Notice also that $\rec(Q_\epsilon) = \rec(\cl(Q)) = R$ for all $\epsilon \in \R$ (they are all given by $\{r \in \R^n: Ar \leq 0\}$). Moreover, since $\cl(Q)$ is pointed, $Q_\epsilon$ is pointed for all $\epsilon \in \R$. Also, there exists a large enough natural integer $M$ such that the box $[-M, M]^n$ contains, in its interior, all the vertices of $Q_\epsilon$ for every $0 \leq \epsilon \leq 1$. 

Define $P' := Q_0 \cap [-M,M]^n$ which is a regular polytope. By standard real analysis arguments, $\inte(P') = \inte(Q_0) \cap \inte([-M,M]^n)$. Thus, we have $\inte(P') + R \subseteq \inte(Q_0) + R = Q$. Next, observe that for any $\epsilon > 0$, $Q_\epsilon \subseteq \inte(P') + R$ because every vertex of $Q_\epsilon$ is in the interior of $Q_0$, as well as in the interior of $[-M, M]^n$ by construction of $M$, and so every vertex of $Q_\epsilon$ is contained in $\inte(P')$. Since $Q = \cup_{0 < \epsilon \leq 1} Q_\epsilon$, we obtain that $Q \subseteq \inte(P') + R$. Putting both inclusions together, we obtain that $Q = \inte(P') + R$. Since $P := \inte(P')$ is a relatively open polytope, we have the desired result. 

The rational version of the proof is along the same lines, where we restrict to rational data.
\end{proof}

We also observe below that generalized polyhedra are closed under Minkowski sums (up to finite unions). 

\begin{lemma}\label{lem:mink-sum-gen-poly} Let $P, Q$ be generalized polyhedra, then $P + Q$ is a finite union of relatively open polyhedra. If $P$ and $Q$ are rational, then $P + Q$ is a finite union of rational relatively open polyhedra.
\end{lemma}
\begin{proof} By \cref{lem:genpolyIsUnionRelOpPoly}, it suffices to prove the result for relatively open polyhedra $P$ and $Q$. The result then follows from the fact that $\relint(P+Q) = \relint(P) + \relint(Q) = P + Q$, where the second equality follows from~\cite[Corollary 6.6.2]{Rockafellar1970}. 

{The rational version follows along similar lines by using the rational version of \cref{lem:genpolyIsUnionRelOpPoly}.}
\end{proof}

%
%

We now prove a preliminary result on the path to generalizing \cref{thm:JeroslowLowe} to generalized polyhedra.
\begin{lemma}\label{lem:genisrelOp}
    Let $Q\subseteq \R^n\times\R^d$ be a { rational} generalized polyhedron. {Then $Q\cap(\Z^n\times \R^d)$ is a union of finitely many sets, each of which is the Minkowski sum of a relatively open rational polytope and a rational monoid.}
\end{lemma}
\begin{proof}
    By \cref{lem:genpolyIsUnionRelOpPoly}, it is sufficient to prove the theorem where $Q$ is relatively open. By \cref{lem:MinkWeyl}, we can write $Q=P+R$ where $P$ is a { rational} relatively open polytope and $R$ is a cone {generated by finitely many rational vectors}. Set $T = P+X$, where $X = \left \lbrace \sum_{i=1}^k \lambda_i r^i: 0\leq \lambda_i\leq 1\right \rbrace$ where $r^i$ are the extreme rays of $R$ whose coordinates can be chosen as integers. For $u+v\in Q = P+R$, where $u\in P,\,v\in R$; let $v = \sum_{i=1}^{k}\mu_ir^i$. Define  $\gamma_i = \floor{\mu_i}$ and $\lambda_i = \mu_i - \gamma_i$. So, $u+v = (u+\sum_{i=1}^k\lambda_ir^i) + \sum_{i=1}^k\gamma_ir^i$, where the term in parentheses is contained in $T$ and since $\gamma_i \in \Z_+$, the second term is in a monoid generated by the extreme rays of $R$.
    Thus, we have $Q\cap(\Z^n\times\R^d) = (T\cap (\Z^n\times\R^d)) + \inte\cone(r^1,\ldots,r^k)$. Since $T$ is a finite union of {rational} relatively open {\em polytopes} by \cref{lem:mink-sum-gen-poly}, $T\cap (\Z^n\times\R^d)$ is a finite union of {rational} relatively open polytopes.
\end{proof}

The following is an analog of Jeroslow and Lowe's fundamental result (\cref{thm:JeroslowLowe}) to the generalized polyhedral setting.

\begin{theorem}\label{cor:jeroslow-lowe-generalized}
The following are equivalent:
    \begin{enumerate}[(1)]
        \item{$S\in {\GDMILPR}(\Q)$}, 
        \item{$S$ is a finite union of sets, each of which is the Minkowski sum of a rational relatively open polytope and a rational finitely generated monoid}, and 
        \item{$S$ is a finite union of sets, each of which is the Minkowski sum of a rational generalized polytope and a rational finitely generated monoid}.
    \end{enumerate}
\end{theorem}


\begin{proof}
    $(1)\implies(2)$: Observe from \cref{lem:ProjIsRelOpen} that a { (rational)} linear transform of a { (rational)} relatively open polyhedron is a { (rational)} relative open polyhedron, and by definition of a { (rational)} monoid, a { (rational)} linear transform of a { (rational)} monoid is a { (rational)} monoid. Now from \cref{lem:DMIRisFiniteUnion,lem:genisrelOp}, the result follows. 

\noindent $(2)\implies(3)$: This is trivial since every { (rational)} relatively open polyhedron is a { (rational)} generalized polyhedron.
  
\noindent $(3)\implies(1)$: This follows from the observation that {the Minkowski sum of a rational generalized polytope and a rational monoid} is a {rational} generalized mixed-integer representable set. A formal proof could be constructed following the proof of \cref{thm:JeroslowLowe} given in the original \citet{Jeroslow1984} or \cite[Theorem~4.47]{Conforti2014}. These results are stated for the case of regular polyhedra but it is straightforward to observe that their proofs equally apply to the generalized polyhedra setting with only superficial adjustments. We omit those minor details for brevity. 
\end{proof}

\begin{remark}\label{rem:rationality-assumption-1} {Notice that the equivalence of (1) and (3) in~\cref{cor:jeroslow-lowe-generalized} is an analog of Jeroslow and Lowe's~\cref{thm:JeroslowLowe}. Moreover, the rationality assumption cannot be removed from \cref{lem:genisrelOp}, and hence cannot be removed from \cref{cor:jeroslow-lowe-generalized}. This is one of the places where the rationality assumption plays a crucial role; see also \cref{rem:rationality-assumption-2}.}\end{remark}

Now we prove that if we intersect sets within the family of generalized MILP-representable sets, then we remain in that family. 

\begin{lemma}\label{lem:MIintersect}
	Let $S$ and $T$ be sets in $\R^n$ where $S,\, T \in {\GMILPR}$. Then $S \cap T \in {\GMILPR}$. {The rational version holds, i.e., one can replace ${\GMILPR}$ by ${\GMILPR}(\Q)$ in the statement.}
\end{lemma}
\begin{proof}

  Using the same construction in \cref{lem:DMIRisFiniteUnion}, one can assume that there is a common ambient space $\R^{n_1}\times \R^{n_2}$ and a common linear transformation $L$ such that
  \begin{align*}
  S &= \left \lbrace L(x): x \in\R^{n_1}\times\Z^{n_2}: A^1x \leq b^1,\,\hat A^1x < \hat b^1\right \rbrace \\ 
  T &= \left \lbrace L(x): x \in\R^{n_1}\times\Z^{n_2}: A^2x \leq b^2,\,\hat A^2x < \hat b^2\right \rbrace.
  \end{align*}
  Now, it is easy to see $S \cap T \in {\GMILPR}$.
  
  {The rational version follows by restricting the linear transforms and data to be rational.}
\end{proof}


An immediate corollary of the above lemma is that the class ${\GDMILPR}$ is closed under finite intersections.

\begin{lemma}\label{lem:DMIIntersect}
	Let $S$ and $T$ be sets in $\R^n$ where $S,\, T \in {\GMILPR}$. Then $S\cap T \in {\GDMILPR}$. {The rational version holds, i.e., one can replace ${\GDMILPR}$ by ${\GDMILPR}(\Q)$ in the statement.}
\end{lemma}
\begin{proof}
	By \cref{lem:DMIRisFiniteUnion}, $S = \bigcup_{i=1}^k S_i$ and $T = \bigcup_{i=1}^\ell T_i$ where $S_i,\,T_i \in {\GMILPR}$. Now $S\cap T = \left (\bigcup_{i=1}^k S_i \right ) \cap \left (\bigcup_{i=1}^\ell T_i \right )$ $= \bigcup_{i=1}^{k}\bigcup_{j=1}^{\ell}\left ( S_i\cap T_j\right )$. But from \cref{lem:MIintersect}, $S_i\cap T_j \in {\GMILPR}$. Then the result follows from \cref{lem:DMIRisFiniteUnion}.
	
	{The rational version follows from the rational versions of \cref{lem:DMIRisFiniteUnion,lem:MIintersect}.}
\end{proof}

To understand the interaction between generalized polyhedra and monoids, we review a few standard terms and results from lattice theory. We refer the reader to \citep{Schrijver1998,Conforti2014,Barvinok2002} for more comprehensive treatments of this subject.

%
%

%
%

\begin{Def}[Lattice]
	Given a set of linearly independent vectors $d^1,\ldots,d^r\in\R^n$, the lattice generated by the vectors is the set
	\begin{equation}
		\Lambda\quad =\quad \left \lbrace x: x=\sum_{i=1}^r \lambda_id^i, \lambda_i \in \Z\right \rbrace
	\end{equation}
	We call the vectors $d^1,\dots,d^r$ as the generators of the lattice $\Lambda$ and denote it by $\Lambda = Z(d^1,\dots,d^r)$.
\end{Def}

Note that the same lattice $\Lambda$ can be generated by different generators.

\begin{Def}[Fundamental parallelepiped]
	Given $\Lambda = Z(d^1,\ldots,d^r) \subseteq\R^n$, we define {\em the fundamental parallelepiped} of $\Lambda$ (with respect to the generators $d^1,\ldots,d^r$) as the set 
	\begin{equation*}
		\Pi_{\{d^1, \ldots, d^r\}} \quad:=\quad \left\{ x\in\R^n : x = \sum_{i=1}^r\lambda_id^i,\, 0\leq \lambda_i < 1 \right\}.
	\end{equation*}
\end{Def}

We prove the following two technical lemmata, which are crucial in proving that {${\GDMILPR}(\Q)$} is an algebra. The lemmata prove that $(P+M)^c$, where $P$ is a polytope and $M$ is a { finitely generated} monoid, is in ${\GDMILPR}$ {(and a corresponding rational version is true)}. The first lemma proves this under the assumption that $M$ is generated by linearly independent vectors. The second lemma uses this preliminary results to prove it for a general monoid.


The proof of the lemma below where $M$ is generated by linearly independent vectors is based on the following key observations. 
\begin{enumerate}
	\item[(i)] {If the polytope $P$ is contained in the fundamental parallelepiped $\Pi$ of the lattice generated by the monoid $M$, then the complement of $P+M$ is just $(\Pi\setminus P) + M$ along with everything outside $\cone(M)$. 
	\item[(ii)] {The entire lattice can be written as a disjoint union of finitely many cosets with respect to an appropriately chosen sublattice. The sublattice is chosen such that its fundamental parallelepiped contains $P$ (after a translation). Then combining the finitely many cosets with the observation in (i), we obtain the result.} }
\end{enumerate}
The first point involving containment of $P$ inside $\Pi$ is needed to avoid any overlap between distinct translates of $P$ in $P+M$. The linear independence of the generators of the monoid is important to be able to use the fundamental parallelepiped in this way. The proof also has to deal with the technicality that the monoid (and the lattice generated by it) need not be full-dimensional.

\begin{lemma}\label{lem:PandLinIndMonoid}
	Let $M\subseteq\R^n$ be a monoid generated by a {\em linearly independent} set of vectors $\mathcal {M} = \{m^1,\ldots,m^k\}$ and let $P$ be a generalized polytope. Then $(P+M)^c \subseteq {\GDMILPR}$. {Moreover, if $P$ and $M$ are both rational, then $(P+M)^c \subseteq {\GDMILPR}(\Q)$.}
\end{lemma}
\begin{proof}
    Suppose $k\leq n$. We now choose vectors $\tilde m^{k+1},\ldots, \tilde m^n$, a scaling factor $\alpha\in \Z_+$ and a translation vector $f \in \R^n$ such that the following all hold:
    \begin{enumerate}
    \item[(i)] $\tilde {\mathcal M}:= \left\{ m^1,\ldots,m^k,\,\tilde m^{k+1},\ldots, \tilde m^n \right\}$ forms a basis of $\R^n$.
    \item[(ii)] $\left\{ \tilde m^i \right\}_{i=k+1}^n$ are orthogonal to each other and each is orthogonal to the space spanned by $\mathcal M$.
    \item[(iii)] $f + P$ is contained in the fundamental parallelepiped defined by the vectors $\bar{\mathcal M} := \alpha \mathcal M \cup \left\{ \tilde m^{k+1},\ldots, \tilde m^n \right\}$.
    \end{enumerate}
     
     Such a choice is always possible because of the boundedness of $P$ and by utilizing the Gram-Schmidt orthogonalization process. Since we are interested in proving inclusion in $\GDMILPR$, which is closed under translations, we can assume $f=0$ without loss of generality.
     \medskip

Define $\tilde\Lambda := Z(\tilde {\mathcal M})$ and $\tilde M:= \inte\cone(\tilde {\mathcal M})$. Define $\bar\Lambda  := Z \left( \bar {\mathcal M} \right)  \subseteq \tilde\Lambda$ and $\bar M := \inte\cone(\bar {\mathcal{M}})\subseteq \tilde M$. Moreover, linear independence of $\tilde {\mathcal{M}}$ and $\bar{\mathcal M}$ implies that $\tilde M = \tilde\Lambda \cap \cone ( \tilde {\mathcal M} )$ and $\bar M = \bar\Lambda \cap \cone ( \bar {\mathcal M} )$. All of these together imply 
\vskip 5pt
\noindent{\em Claim~\ref{lem:PandLinIndMonoid}.1:} $\bar M= \bar{\Lambda}\cap \tilde M$.
\vskip 5pt
\noindent {\em Proof of Claim~\ref{lem:PandLinIndMonoid}.1:} $\bar M= \bar\Lambda \cap \cone ( \bar {\mathcal M} ) = \bar\Lambda \cap \cone ( \tilde {\mathcal M} ) = (\bar\Lambda \cap \tilde \Lambda) \cap \cone ( \tilde {\mathcal M} ) = \bar\Lambda \cap \left(\tilde \Lambda \cap \cone ( \tilde {\mathcal M} )\right) = \bar{\Lambda}\cap \tilde M$, where the second equality follows from the fact that $\cone ( \bar {\mathcal M} ) = \cone ( \tilde {\mathcal M})$.

 \medskip
    
\noindent{\em Claim~\ref{lem:PandLinIndMonoid}.2: }$\Pi_{\bar{\mathcal M}} + \bar M = \cone(\tilde{\mathcal{M}})$. Moreover, given any element in $x \in \cone(\tilde{\mathcal{M}})$, there exist unique $u\in\Pi_{\bar{\mathcal M}}$ and $v\in\bar M$ such that $x = u+v$.
\vskip 5pt
   \noindent {\em Proof of Claim~\ref{lem:PandLinIndMonoid}.2:} Suppose $u \in \Pi_{\bar{\mathcal{M}}}$ and $v\in\bar M$, then both $u$ and $v$ are non-negative combinations of elements in $\bar{\mathcal{M}}$. So clearly $u+v$ is also a non-negative combination of those elements. This proves the forward inclusion. To prove the reverse inclusion, let $x \in \cone  ( \tilde {\mathcal{M}})$. Then $x = \sum_{i=1}^k\lambda_im^i + \sum_{i=k+1}^n\lambda_i\tilde m^i $ where $\lambda_i\in\R_+$. But now we can write 
    \begin{equation*}
        x \quad = \quad \left (\sum_{i=1}^k\floor{\frac{\lambda_i}{\alpha}}\alpha m^i + \sum_{i=k+1}^n\floor{\lambda_i}\tilde m^i  \right ) + \left ( \sum_{i=1}^k\left(\frac{\lambda_i}{\alpha}-\floor{\frac{\lambda_i}{\alpha}}\right)\alpha m^i + \sum_{i=k+1}^n(\lambda_i-\floor{\lambda_i})\tilde m^i \right ), 
    \end{equation*}
    where the term in the first parentheses is in $\bar M$ and the term in the second parentheses is in $\Pi_{\bar{\mathcal{M}}}$. Uniqueness follows from linear independence arguments, thus proving the claim.
    \medskip

\noindent{\em Claim~\ref{lem:PandLinIndMonoid}.3: } $\left( \Pi_{\bar{\mathcal M} } + \bar M \right) \setminus(P+\bar M) = (\Pi_{\bar{\mathcal M} }\setminus P) + \bar M$.
\vskip 5pt
\noindent {\em Proof of Claim~\ref{lem:PandLinIndMonoid}.3: } Note that, by construction, $P\subseteq \Pi_{\bar{\mathcal{M}}}$. By the uniqueness result in Claim~\ref{lem:PandLinIndMonoid}.2, $(\Pi_{\bar{\mathcal{M}}} +u)\cap (\Pi_{\bar{\mathcal{M}}}+v)=\emptyset$ for $u,\,v\in \bar M$ and $u\neq v$. Thus we have, $x = u+v= u'+v'$ such that $u\in \Pi_{\bar{\mathcal{M}}},\,u'\in P,\,v,\,v'\in \bar M$ implies $v = v'$. Then the claim follows.
    \medskip
    
    Also, there exists a {\em finite} set $S$ such that $\tilde\Lambda = S+\bar\Lambda $ (for instance, $S$ can be chosen to be $\Pi_{\bar{\mathcal M} }\cap \tilde\Lambda$) \citep[Theorem VII.2.5]{Barvinok2002}.  So we have
\begin{align}
	P + \tilde M \quad&=\quad P + (\tilde \Lambda \cap \tilde M) \notag\\ 
  \quad&=\quad \bigcup_{s\in S} \left( P+ \left( (s+\bar\Lambda )\cap \tilde M \right)  \right) \notag \\
    \quad&=\quad \bigcup_{s\in S}\left ( P + (\bar \Lambda \cap \tilde M)+s\right ) \notag \\
    \quad&=\quad \bigcup_{s\in S}\left ( P + \bar M+s\right ),\label{eq:cosetUnion}
\end{align}
where the last equality follows from Claim~\ref{lem:PandLinIndMonoid}.1.
The intuition behind the above equation is illustrated in \cref{fig:PandLinIndMonoid}. 
\begin{figure}
\captionsetup[subfigure]{justification=centering}
	\begin{subfigure}{0.45\textwidth}\centering
		\includegraphics[width=\linewidth]{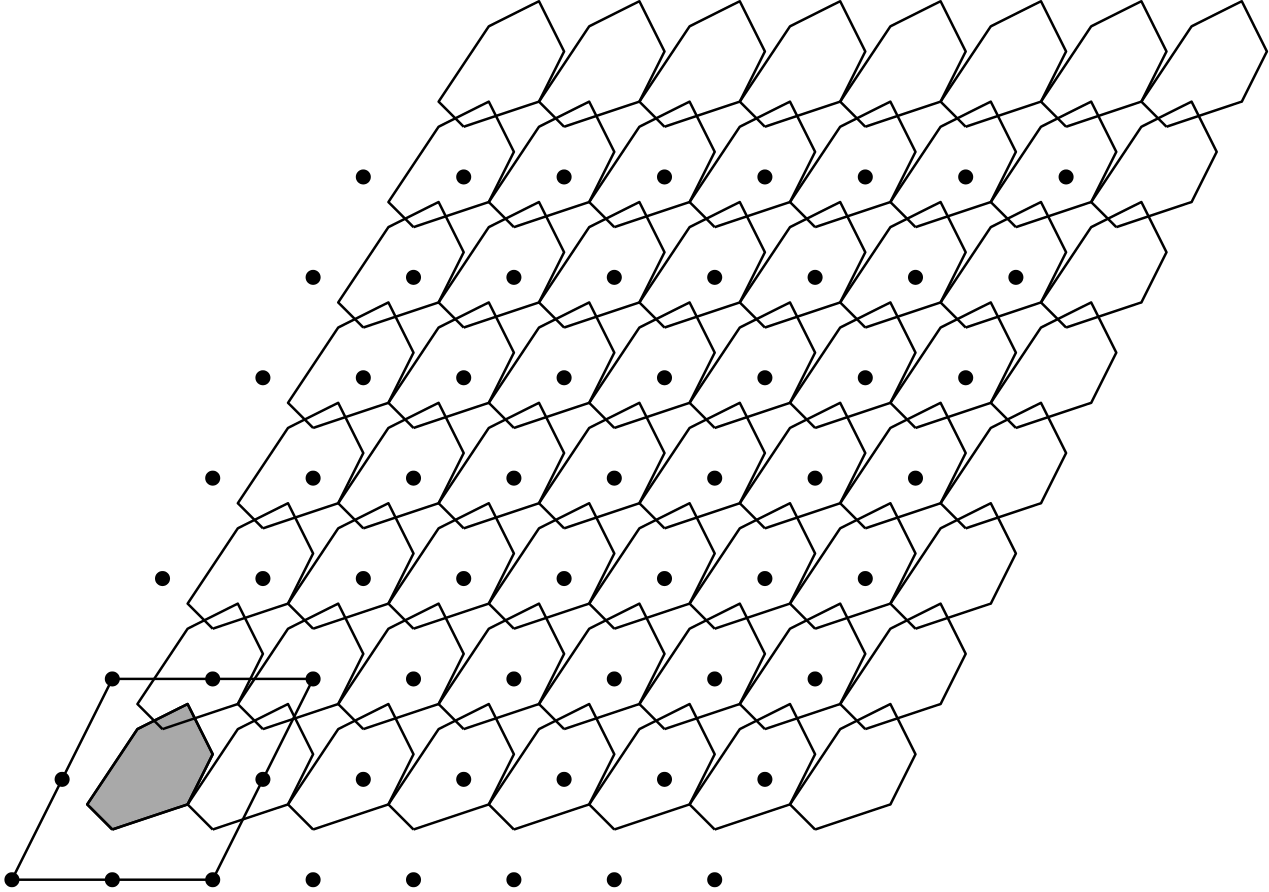}
		\caption{$P$ is shown in gray and the translates of $P$ along $\tilde M$ are shown. The fundamental parallelepiped $\Pi_{\bar\Lambda }$ is also shown to contain $P$.}
		\label{fig:PandLinIndMonoida}
	\end{subfigure}
	\begin{subfigure}{0.45\textwidth}\centering
		\includegraphics[width=\linewidth]{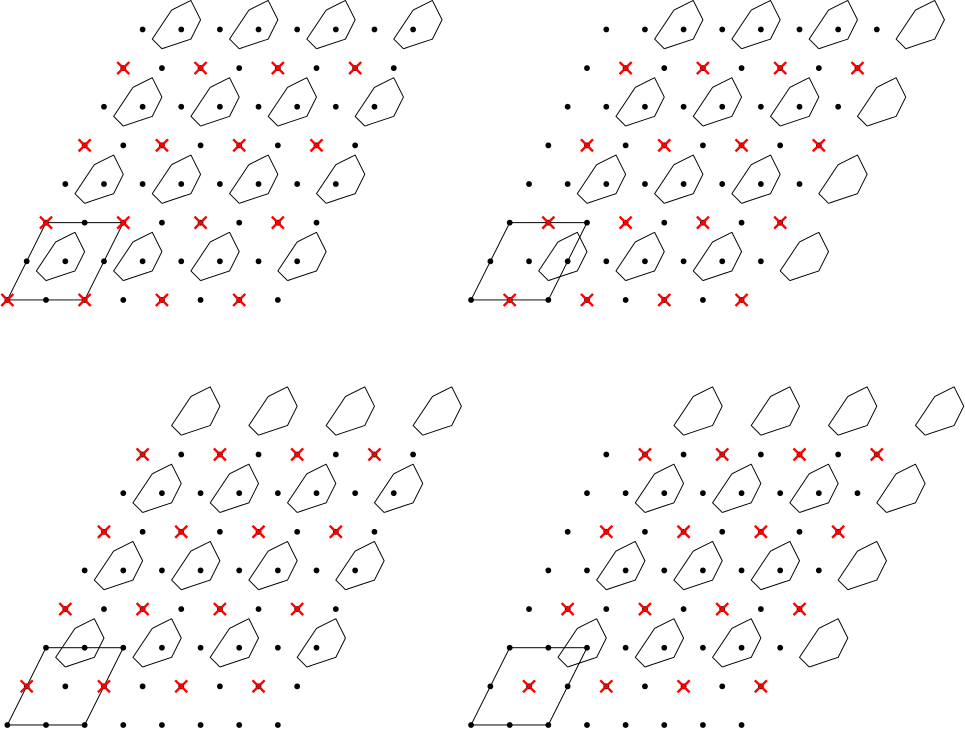}
        \caption{$P+((s+\bar \Lambda) \cap \tilde M)$ is shown for each $s\in S$. The red crosses correspond to the translation of $\bar M$ along each $s\in S$. The union of everything in \cref{fig:PandLinIndMonoidb} is \cref{fig:PandLinIndMonoida}.}
		\label{fig:PandLinIndMonoidb}
	\end{subfigure}
    \caption{Intuition for the set $S$ such that $\bar\Lambda + S = \tilde \Lambda$.}
	\label{fig:PandLinIndMonoid}
\end{figure}
    We will first establish that $(P+\tilde M)^c \in {\GDMILPR}$. By taking complements in~\eqref{eq:cosetUnion}, we obtain that $(P+\tilde M)^c = \bigcap_{s\in S}\left( P+ \bar M + s \right)^c $. But from \cref{lem:MIintersect}, and from the finiteness of $S$, if we can show that $( P+ \bar M + s)^c$ is in ${\GDMILPR}$ for every $s \in S$, then we would have established that $(P+\tilde M)^c \in {\GDMILPR}$.

Since each of the finite $s\in S$ induce only translations, without loss of generality, we can only consider the case where $s=0$. Since we have $P+\bar M \subseteq \cone (\tilde {\mathcal{M}})$, we have
\begin{align}
	\left( P+\bar M \right) ^c \quad&=\quad \cone(\tilde {\mathcal M})^c \cup \left( \cone(\tilde {\mathcal M})\setminus (P+\bar M) \right) \notag  \\
    \quad&=\quad  \cone(\tilde {\mathcal M})^c \cup \left( \left( \Pi_{\bar{\mathcal M} } + \bar M \right) \setminus(P+\bar M)   \right), \label{eq:last-bit} 
\end{align}
which follows from Claim~\ref{lem:PandLinIndMonoid}.2. Continuing from \eqref{eq:last-bit}:
\begin{equation*}
\left( P+\bar M \right) ^c  \quad =\quad \cone(\tilde {\mathcal M})^c \cup \left( (\Pi_{\bar{\mathcal M} }\setminus P) + \bar M \right), \label{eq:good-looking-union}
\end{equation*}
which follows from Claim~\ref{lem:PandLinIndMonoid}.3.

The first set $\cone(\tilde {\mathcal M})^c$ in \eqref{eq:good-looking-union} belongs to ${\GDMILPR}$ since the complement of a cone is a finite union of generalized polyhedra. In the second set $(\Pi_{\bar{\mathcal M} }\setminus P) + \bar M$, $\Pi_{\bar{\mathcal M} }$ and $P$ are generalized polytopes, and hence $\Pi_{\bar{\mathcal M} }\setminus P$ is a finite union of generalized polytopes, $(\Pi_{\bar{\mathcal M} }\setminus P) + \bar M$ is a set in ${\GDMILPR}$. Thus, $(P+\bar M)^c \in {\GDMILPR}$ (note that \cref{lem:DMIRisFiniteUnion} shows that ${\GDMILPR}$ is closed under unions). 
\medskip

We now finally argue that $(P + M)^c$ belongs to ${\GDMILPR}$. Let $A^1:=(P+\tilde M)^c$.

For each vector $\tilde m^i$ for $i=k+1,\ldots,n$ added to form $\tilde{\mathcal M}$ from $\mathcal M$, define $H^i$ as follows:
\begin{equation*}
H^i \quad=\quad \left\{ x: \langle \tilde m^i,\, x \rangle \geq \Vert\tilde m^i \Vert_2^2 \right\}.	
\end{equation*}
    Now let $A^2 := \bigcup_{i=k+1}^n H^i$. Note that $A^2$ is a finite union of halfspaces and hence $A^2 \in {\GDMILPR}$. We claim $(P+M)^c = A^1\cup A^2$. This suffices to complete the argument since we have shown $A^1$ and $A^2$ are in ${\GDMILPR}$ and thus so is their union. 

    First, we show that $A^1 \cup A^2 \subseteq (P+M)^c$, i.e., $P+M \subseteq A_1^c \cap A_2^c$. Let $x\in P+M$. Since $M \subseteq \tilde M$ we have $x \in P+\tilde M$. Thus, $x \not\in A^1$. Further, since $x \in P + M$ we may write  $x = u +v$ with $u\in P$ and $v\in M$ where $u=\sum_{i=1}^k\mu_i\alpha m^i + \sum_{i=k}^n \mu_i\tilde m^i$,\, $v = \sum_{j=1}^k\lambda_jm^j$ with $0\leq\mu < 1$ and $\lambda_j \in \Z_+$, since $P \subseteq \Pi_{\bar {\mathcal M}}$. So for all $i$, $ \langle \tilde m^i,\, u \rangle = \mu_i \Vert\tilde m^i \Vert_2^2 < \Vert\tilde m^i \Vert_2^2 $. This is because we have $\tilde m^i$ is orthogonal to every vector $m^j$ and $\tilde m^j$ for $i\neq j$. Hence, $\langle \tilde m^i,\, u+ v \rangle < \Vert\tilde m^i \Vert_2^2 $. This follows because $\tilde m^i$ is orthogonal to the space spanned by the monoid $M\ni v$. Thus $x \not\in A^2$. So we now have $P+M \subseteq A_1^c \cap A_2^c$ and so  $A^1 \cup A^2 \subseteq (P+M)^c$. 

    Conversely, suppose $x \not\in P+M$. If, in addition, $x\not\in P+\tilde M$ then $x\in A^1$ and we are done. However, if $x = u + v\in P+ (\tilde M\setminus M)$ with $u\in P$ and $v\in \tilde M\setminus M$. This means $v = \sum_{j=1}^k\lambda_jm^j + \sum_{j=k+1}^n \lambda_j\tilde m^j$ with $\lambda_j \in \Z_+$ for all $j = 1,\ldots,n$ and  $\lambda_{\bar j} \geq 1$  for some $\bar j \in \{k+1,\ldots,n\}$ and we can write $u = \sum_{j=1}^k \mu_j\alpha m^j + \sum_{j=k+1}^n \mu_j \tilde m^j $ with $0\leq \mu \leq 1$. So $\langle \tilde m^i,\, u \rangle = \langle \tilde m^i,\, \mu^i\tilde m^i \rangle \geq 0$ and $\langle \tilde m^{\bar j},\, v \rangle = \langle \tilde m^{\bar j},\, \lambda_{\bar j}\tilde m^{\bar j} \rangle > \Vert m^{\bar j}\Vert _2^2$. So $u+v \in H^{\bar j}\subseteq A^2$. Thus we have the reverse containment and hence the result.
    
    { The rational version follows along similar lines.}
\end{proof}

\begin{lemma}\label{lem:MILPRcompl}
	Let $P\subseteq\R^n$ be a { rational} generalized polytope and $M\in\R^n$ be a {rational, finitely generated} monoid. Then $S = (P+M)^c \in {\GDMILPR}(\Q)$.
\end{lemma}
\begin{proof}
Define $C:= \cone(M)$. Consider a triangulation $C = \bigcup_i C_i$, where each $C_i$ is simplicial. Now, $M_i := M \cap C_i$ is a monoid for each $i$ (one simply checks the definition of a monoid) and moreover, it is a pointed monoid because $C_i$ is pointed and $\cone(M_i) = C_i$ since every extreme ray of $C_i$ has an element of $M_i$ on it. Observe that $M = \bigcup_i M_i$.

    By Theorem 4, part 1) in \citet{jeroslow1978some}, each of the $M_i$ are finitely generated. By part 3) of the same theorem, each $M_i$ can be written as $M_i = \bigcup_{j = 1}^{w_i} (p^{i,j} + \bar{M}_i)$ for some finite vectors $p^{i,1}, \ldots, p^{i,w_i} \subseteq M_i$, where $\bar{M}_i$ is the monoid generated by the elements of $M_i$ lying on the extreme rays of $C_i$.
Now, 
\begin{equation*}
P+M \quad=\quad \bigcup_{i} (P + M_i) \quad=\quad \bigcup_{i} \bigcup_{j=1}^{w_i} (P+ (p^{i,j} + \bar{M}_i)).
\end{equation*}
Thus by \cref{lem:DMIIntersect}, it suffices to show that $(P+ (p^{i,j} + \bar{M}_i))^c$ is in ${\GDMILPR}$. Since $\bar{M}_i$ is generated by linearly independent vectors, we have our result from \cref{lem:PandLinIndMonoid}.
\end{proof}
\begin{remark}\label{rem:rationality-assumption-2}
{We do not see a way to remove the rationality assumption in \cref{lem:MILPRcompl}, because it uses Theorem 4 in \citet{jeroslow1978some} that assumes that the monoid is rational and finitely generated. This is the other place where rationality becomes a crucial assumption in the analysis (see also \cref{rem:rationality-assumption-1}).}
\end{remark}

\begin{lemma}\label{cor:MILPRcompl}
    If $S \in {{\GDMILPR}(\Q)}$ then $S^c \in {{\GDMILPR}(\Q)}$.
\end{lemma}
\begin{proof}
	{By \cref{cor:jeroslow-lowe-generalized}, $S$ can be written as a finite union} $S = \bigcup_{j=1}^\ell S_j$, with $S_j = P_j+M_j$, where{ $P_j$ is a rational generalized polytope} and $M_j$ is a {rational,} finitely generated monoid. {Observe $S_j^c = (P_j + M_j)^c$, which by Lemma \ref{lem:MILPRcompl}, is in ${\GDMILPR}(\Q)$.} Now by De Morgan's law, $S^c = \left ( \bigcup_j S_j\right )^c = \bigcap_j S_j^c$. By \cref{lem:DMIIntersect}, ${\GDMILPR}(\Q)$ is closed under intersections, and we have the result.
 \end{proof}


\begin{proof}[Proof of \cref{thm:algebra}]
	We recall that a family of sets $\mathscr{F}$ is an algebra if the following two conditions hold. (i) $S\in\mathscr F \implies S^c \in \mathscr F$ and (ii) $S,\,T\in\mathscr F \implies S \cup T \in \mathscr F$. For the class of interest, the first condition is satisfied from \cref{cor:MILPRcompl} and noting that the complement of finite unions is a finite intersection of the complements and that the family $\GDMILPR(\Q)$ is closed under finite intersections by~\cref{lem:DMIIntersect}. The second condition is satisfied by \cref{lem:DMIRisFiniteUnion} which shows that ${\GDMILPR(\Q)}$ is the same as {\em finite unions} of sets in ${\GMILPR(\Q)}$.
\end{proof}

\subsection{Value function analysis}\label{ss:value-function}


We now discuss the three classes of functions defined earlier, namely, Chv\'atal functions, Gomory functions and Jeroslow functions, and show that their sublevel, superlevel and level sets are all elements of ${\GDMILPR(\Q)}$. This is crucial for studying bilevel integer programs via a value function approach to handling the lower-level problem using the following result.

\begin{theorem}[Theorem 10 in \citet{blair1995closed}]\label{thm:MIPvalfun}
    For every rational mixed integer program $\max_{x,\,y} \{ c^\top x + d^\top y : Ax + By =b;\, (x,\,y)\geq 0,\, x\in\Z^m\}$ there exists a Jeroslow function $J$ such that if the program is feasible for some $b$, then its optimal value is $J(b)$.
\end{theorem}

We now show that the sublevel, superlevel, and level sets of Chv\'atal functions are all in ${\GDMILPR}(\Q)$.

\begin{lemma}\label{lem:ChvatalisDMILP}
    Let $\psi:\R^n\mapsto\R$ be a Chv\'atal function. Then (i) $\left \lbrace x:\psi(x) \geq 0\right \rbrace$, (ii) $\left \lbrace x:\psi(x) \leq 0\right \rbrace$, (iii) $\left \lbrace x:\psi(x) = 0\right \rbrace$, (iv) $\left \lbrace x:\psi(x) < 0\right \rbrace$, and (v) $\left \lbrace x:\psi(x) > 0\right \rbrace$ are all in ${\GDMILPR}(\Q)$.
\end{lemma}
\begin{proof}
    (i) 
    We prove by induction on the order $\ell\in\Z^+$ used in the binary tree representation of $\psi$ (see \cref{Def:ChvFun}). For $\ell = 0$, a Chv\'atal function is a rational linear inequality and hence $\left \lbrace x:\psi(x) \geq 0\right \rbrace$ is a rational halfspace, which is clearly in ${\GDMILPR}(\Q)$. Assuming that the assertion is true for all orders $\ell \leq k$, we prove that it also holds for order $\ell = k+1$.
    By \cite[Theorem.~4.1]{Basu2016}, we can write $\psi(x) = \psi_1(x) + \floor{\psi_2(x)}$ where $\psi_1$ and $\psi_2$ are Chv\'atal functions with representations of order no greater than $k$. Hence,
    \begin{align*}
        \left \lbrace x:\psi(x)\geq 0\right \rbrace \quad&=\quad \left \lbrace x: \psi_1(x) + \floor{\psi_2(x)} \geq 0\right \rbrace\\
        \quad&=\quad \left \lbrace x: \exists y \in \Z,\, \psi_1(x)+y \geq 0,\, \psi_2(x) \geq y\right \rbrace.
    \end{align*}
    We claim equivalence because, suppose $\bar x$ is an element of the set in RHS with some $\bar y \in \Z$, $\bar y$ is at most $\floor{\psi_2(x)}$. So if $\psi_1(\bar x) + \bar y\geq 0$, we immediately have $\psi_1(\bar x) + \floor{\psi_2(\bar x)} \geq0$ and hence $\bar x$ is in the set on LHS. Conversely, if $\bar x$ is in the set on LHS, then choosing $\bar y = \floor{\bar x}$ satisfies all the conditions for the sets in RHS, giving the equivalence. Finally, observing that the RHS is an intersection of sets which are already in ${\GDMILPR}(\Q)$ by the induction hypothesis and the fact that ${\GDMILPR}(\Q)$ is an algebra by \cref{thm:algebra}, we have the result.  

    (ii) By similar arguments as (i), the statement is true for $\ell = 0$. For positive $\ell$, we proceed by induction using the same construction. Now,
    \begin{align*}
    \left \lbrace x:\psi(x) \leq 0\right \rbrace  \quad&=\quad \left \lbrace x:\psi_1(x)  + \floor{\psi_2(x)} \leq 0\right \rbrace \\
     \quad& =\quad \left \lbrace x: \exists y\in\Z,\, \psi_1(x)+y \leq 0,\, \psi_2(x)\geq y,\,\psi_2(x) < y+1 \right \rbrace.
    \end{align*}

    The last two conditions along with integrality on $y$ ensures $y = \floor{\psi_2(x)}$. Note that $\left \lbrace x : \psi_2(x) -y \geq 0\right \rbrace$ is in ${\GDMILPR}(\Q)$ by (i). Similarly $\left \lbrace x:\psi_2(x) -y-1 \geq 0 \right \rbrace \in {\GDMILPR}(\Q)$. Since ${\GDMILPR}(\Q)$ is an algebra (cf. \cref{thm:algebra}), its complement is in ${\GDMILPR}(\Q)$ and hence we have $\left \lbrace x: \psi_2(x) < y+1\right \rbrace \in \GDMILPR(\Q)$. Finally from the induction hypothesis, we have $\left \lbrace x: \psi_1(x) + y \leq 0\right \rbrace \in {\GDMILPR}(\Q)$. Since ${\GDMILPR}(\Q)$ is closed under finite intersections, the result follows.
    
    (iii) Set defined by (iii) is an intersection of sets defined in (i) and (ii).

    (iv)-(v) Sets defined here are complements of sets defined in (i)-(ii).
\end{proof}

\begin{lemma}\label{lem:GomoryisDMILP}
 Let $G:\R^n\mapsto\R$ be a Gomory function. Then (i) $\left \lbrace x:G(x) \geq 0\right \rbrace$, (ii) $\left \lbrace x:G(x) \leq 0\right \rbrace$, (iii) $\left \lbrace x:G(x) = 0\right \rbrace$, (iv) $\left \lbrace x:G(x) < 0\right \rbrace$, (v) $\left \lbrace x:G(x) > 0\right \rbrace$ are all in ${\GDMILPR}(\Q)$.
\end{lemma}

\begin{proof}
    Let $G(x) = \min_{i=1}^k\psi_i(x)$, where each $\psi_i$ is a Chv\'atal function. 
    
    (i) Note that $\left \lbrace x:G(x) \geq 0\right \rbrace = \bigcap_{i=1}^k \left \lbrace x:\psi_i(x) \geq 0\right \rbrace \in \GDMILPR(\Q)$ since each individual set in the finite intersection is in $\GDMILPR(\Q)$ by \cref{lem:ChvatalisDMILP} and $\GDMILPR(\Q)$ is closed under intersections by \cref{lem:DMIIntersect}.
    
    (ii) Note that $G(x) \leq 0$ if and only if there exists an $i$ such that $\psi_i(x)\leq0$. So $\left \lbrace x:G(x) \leq 0\right \rbrace = \bigcup_{i=1}^k \left \lbrace x:\psi_i(x) \leq 0\right \rbrace \in \GDMILPR(\Q)$ since each individual set in the finite union is in ${\GDMILPR}(\Q)$ by \cref{lem:ChvatalisDMILP}, and${\GDMILPR}(\Q)$ is an algebra by \cref{thm:algebra}.

(iii) This is the intersection of sets described in (i) and (ii).

    (iv)-(v) Sets defined here are complements of sets defined in (i)-(ii). 
\end{proof}

\begin{lemma}\label{lem:JeroisDMILP}
    Let $J:\R^n\mapsto\R$ be a Jeroslow function. Then (i) $\left \lbrace x:J(x) \geq 0\right \rbrace$, (ii) $\left \lbrace x:J(x) \leq 0\right \rbrace$, (iii) $\left \lbrace x:J(x) = 0\right \rbrace$, (iv) $\left \lbrace x:J(x) < 0\right \rbrace$, and (v) $\left \lbrace x:J(x) > 0\right \rbrace$ are all in ${\GDMILPR}(\Q)$.
\end{lemma}

\begin{proof}
	(i) Let $J(x) = \max_{i\in \mathcal{I}}  G\left (\floor{x}_{E_i}\right ) + w_i^\top (x - \floor{x}_{E_i})$ be a Jeroslow function, where $G$ is a Gomory function, $\mathcal I$ is a finite set, $\{E_i\}_{i\in \mathcal{I}}$ is set of rational invertible matrices indexed by $\mathcal{I}$, and $\{w_i\}_{i\in \mathcal{I}}$ is a set of rational vectors indexed by $\mathcal{I}$.  Since we have a maximum over finitely many sets, from the fact that ${\GDMILPR}(\Q)$ is an algebra, it suffices to show $\left \lbrace  x: G(\floor{x}_E) + w^\top (x - \floor{x}_E) \geq 0 \right \rbrace \in {\GDMILPR}(\Q)$ for arbitrary $E$, $w$ and Gomory function $G$. Observe that

\begin{align*}
    &\quad\left \lbrace  x: G(\floor{x}_E) + w^\top (x - \floor{x}_E) \geq 0 \right \rbrace 
    \quad = \quad\proj_x\left \lbrace (x, y^1, y^2, y^3) : \begin{array}{c} G(y^1)+y^2 \geq 0\\
    y^1 = \floor{x}_E \\
    y^2 = \langle w,\, y^3 \rangle\\
    y^3 = x-y^1
    \end{array} \right \rbrace 
\end{align*}   
and the set being projected in the right hand side above is equal to the following intersection 

\begin{align*}
       & \left \lbrace (x,y^1,y^2,y^3):\quad G(y^1)+y^2 \geq 0\right \rbrace \\
    \cap\quad& \left \lbrace (x,y^1,y^2,y^3):\quad E^{-1}y^1 = \floor{E^{-1}x} \right \rbrace \\
    \cap\quad&\left \lbrace (x,y^1,y^2,y^3):\quad y^2 = \langle w,\, y^3 \rangle \right \rbrace \\
    \cap\quad&\left \lbrace (x,y^1,y^2,y^3):\quad y^3 = x-y^1 \right \rbrace . 
\end{align*}

Since each of the sets in the above intersection belong to ${\GDMILPR}(\Q)$ by \cref{lem:GomoryisDMILP,lem:ChvatalisDMILP}, and ${\GDMILPR}(\Q)$ is an algebra by \cref{thm:algebra}, we obtain the result.
    
(ii) As in (i), since we have maximum over finitely many sets, from the fact that ${\GDMILPR}(\Q)$ is an algebra (\cref{thm:algebra}), it suffices to show $\left \lbrace  x: G(\floor{x}_E) + w^\top (x - \floor{x}_E) \leq 0 \right \rbrace \in {\GDMILPR}(\Q)$ for arbitrary $E$, $w$ and Gomory function $G$. The same arguments as before pass through, except for replacing the $\geq$ in the first constraint with $\leq$.

(iii) This is the intersection of sets described in (i) and (ii).

    (iv)-(v) Sets defined here are complements of sets defined in (i)-(ii). 
\end{proof}

\begin{proof}[Proof of \cref{thm:ChvGomJer}]
    Follows from \cref{lem:ChvatalisDMILP,lem:GomoryisDMILP,lem:JeroisDMILP}.
\end{proof}

\subsection{General mixed-integer bilevel sets} 
\label{sub:general_mixed_integer_bilevel_sets}


We start by quoting an example from \citet{Koppe2010} showing that the MIBL set need not even be a closed set. This is the first relation in \cref{thm:RepMIBL}, showing a strict containment.

\begin{lemma}\cite[Example~1.1]{Koppe2010}\label{lem:StrictIncl}
The following holds:	
\begin{equation*}
\mathscr T_R^{MIBL}\setminus\mathscr \cl\left ( \mathscr T_R^{MIBL}\right) \neq \emptyset.
\end{equation*}
\end{lemma}
\begin{proof} 
	The following set $T$ is in $\mathscr T_R^{MIBL}\setminus\cl\left ( \mathscr T_R^{MIBL}\right) $:
	\begin{equation*}
		T\quad=\quad \left\{ (x,y)\in\R^2: 0\leq x\leq 1,\; y\in \arg\min_y \left\{ y:y\geq x,\,0\leq y\leq 1, y\in\Z \right\} \right\}.
	\end{equation*}
	By definition, $T\in \mathscr T_R^{MIBL}$. Observe that the bilevel constraint is satisfied only if $y = \ceil{x}$. So $T = \left\{ (0,\,0) \right\}\;\cup\; \left ((0,\,1]\times \left\{ 1 \right\}\right)$. So $T$ is not a closed set. Observing that every set in $\cl\left ( \mathscr T_R^{MIBL}\right)$ is closed, $T\not\in\cl\left ( \mathscr T_R^{MIBL}\right)$.
\end{proof}

Now we prove a lemma that states that rational MIBL-representable sets are in ${\GDMILPR}(\Q)$. 
\begin{lemma}\label{lem:MIBLinDMI}
The following holds: $\mathscr T_R^{MIBL}(\Q) \subseteq \GDMILPR(\Q)$.
\end{lemma}
\begin{proof}
    Recall from \cref{Def:MIBL} an element $S$ of $\mathscr T^{MIBL}(\Q)$ consists of the intersection $S^1 \cap S^2 \cap S^3$ (with rational data). From \cref{thm:MIPvalfun}, $S^2$ can be rewritten as $\left \lbrace (x,\,y) : f^\top  y \geq J(g-Cx) \right \rbrace$ for some Jeroslow function $J$. Thus, from \cref{lem:JeroisDMILP}, $S^2 \in  {\GDMILPR}(\Q)$. Moreover, $S^1,\,S^3 \in  {\GDMILPR}(\Q)$ since they are either rational polyhedra or mixed-integer points in rational polyhedra. Thus, $S=S^1\cap S^2\cap S^3 \in  {\GDMILPR}(\Q)$ by \cref{thm:algebra}, proving the inclusion. This shows that $\mathscr T^{MIBL}(\Q) \subseteq {\GDMILPR}(\Q),$ and by \cref{thm:ReprIncl} the result follows.
\end{proof}

\cref{lem:MIBLinDMI} gets us close to showing $\cl(\mathscr T_R^{MIBL}(\Q)) \subseteq \DMILPR(\Q)$, as required in \cref{thm:RepMIBL}. Indeed, we can immediately conclude from \cref{lem:MIBLinDMI} that $\cl(\mathscr T_R^{MIBL}(\Q)) \subseteq \cl(\GDMILPR(\Q))$. The next few results build towards showing that $\cl(\GDMILPR) = \DMILPR$, and consequently $\cl(\GDMILPR(\Q)) = \DMILPR(\Q)$. The latter is intuitive since closures of generalized polyhedra are regular polyhedra. As we shall see, the argument is a bit more delicate than this simple intuition. We first recall a couple of standard results on the closure operation $\cl(\cdot)$.

\begin{lemma}\label{lem:ClosureUnion}
    If $S_1,\ldots,S_k\in \R^n $ then $
        \cl \left ( \bigcup_{i=1}^n S_i\right ) = \bigcup_{i=1}^n \cl\left ( S_i\right ).$
\end{lemma}

\begin{proof}
    Note $S_j \subseteq \bigcup_iS_i$ for any $j\in[k]$. So $\cl(S_j) \subseteq \cl (\bigcup_i S_i)$. So, by union on both sides over all $j\in[n]$, we have that the RHS is contained in LHS. Conversely, observe that the RHS is a closed set that contains every $S_i$. But by definition, LHS is the inclusion-wise {\em smallest} closed set that contains all $S_i$. So the LHS is contained in the RHS, proving the lemma.
\end{proof}


\begin{lemma}\label{lem:MinkProjec}
	Let $A, B$ be sets such that $A$ is a finite union of convex sets, $\cl(A)$ is compact and $B$ is closed. Then $\cl(A+B) = \cl(A) + B$.
\end{lemma}
\begin{proof}
	For the inclusion $\supseteq$, we refer to Corollary 6.6.2 in \citet{Rockafellar1970}, which is true for arbitrary convex sets $A,\, B$. The result naturally extends using \cref{lem:ClosureUnion} even if $A$ is a finite union of convex sets. For the reverse inclusion, consider $z\in\cl(A+B)$. This means, there exist infinite sequences $\{x^i\}_{i=1}^\infty \subseteq A$ and $\{y^i\}_{i=1}^\infty\subseteq B$, such that the sequence $\{x^i + y^i\}_{i=1}^\infty$ converges to $z$. Now, since $\cl(A)\supseteq A$, $x^i\in\cl(A)$ and since $\cl(A)$ is compact, there exists a subsequence, which has a limit in $\cl(A)$. Without loss of generality, let us work only with such a subsequence $x^i$ and the limit as $x$. Now from the fact that each $y^i\in B$, $B$ is a closed set and the sequence $y^i$ converges to $z-x$, we can say $z-x \in B$. This proves the result, as we wrote $z$ as a sum of $x\in \cl(A)$ and $z-x \in B$.
\end{proof}

\begin{lemma}\label{lem:closure-of-GMILPR-is-DMILPR}
The following holds: $\cl(\GDMILPR) = \DMILPR$. Moreover, $\cl(\GDMILPR(\Q)) = \DMILPR(\Q)$.
\end{lemma}
\begin{proof}
The $\supseteq$ direction is trivial because sets in $\DMILPR$ are closed and a regular polyhedron is a type of generalized polyhedron. For the $\subseteq$ direction, let $S \in \cl(\GDMILPR)$; that is, $S = \cl (\bigcup_{i=1}^k (P_i + M_i))$ for some $P_i$ that are finite unions of generalized polytopes and $M_i$ that are finitely generated monoids. By \cref{lem:ClosureUnion}, this equals $\bigcup_{i=1}^k \cl(P_i + M_i)$. Observe that $P_i$ is a finite union of generalized {\em polytopes} and are hence bounded. Thus their closures are compact. Also $M_i$ are finitely generated monoids and are hence closed. Thus, by \cref{lem:MinkProjec}, we can write this is equal to $\bigcup_{i=1}^k \cl(P_i) + M_i$. But by \cref{thm:JeroslowLowe}, each of these sets $\cl(P_i)+M_i$ is in $\MILPR$. Thus, their finite union is in $\DMILPR$ by \cref{lem:DMIRisFiniteUnion}.

The rational version follows by assuming throughout the proof that the generalized polytopes and monoids are rational.
\end{proof}

The following is then an immediate corollary of \cref{lem:MIBLinDMI,lem:closure-of-GMILPR-is-DMILPR}.

\begin{corollary}\label{cor:closure-of-mibl-is-dmilpr}
The following holds: $\cl(\mathscr T_R^{MIBL}(\Q)) \subseteq \DMILPR(\Q)$.
\end{corollary}

We are now ready to prove the main result of the section.

\begin{proof}[Proof of \cref{thm:RepMIBL}]
    The strict inclusion follows from \cref{lem:StrictIncl}. The equalities $\mathscr {T}_R^{BLP-UI}(\Q) = \DMILPR(\Q)$ and $\mathscr {T}_R^{BLP-UI} = \DMILPR$ are obtained from \cref{lem:DMIPinUIMIBL,lem:UIMIBLinDMIP}. For the equality $\cl(\mathscr T_R^{MIBL}(\Q)) = \DMILPR(\Q)$, the inclusion $\supseteq$ follows from the equality $\mathscr {T}_R^{BLP-UI}(\Q) = \DMILPR(\Q)$ and the facts that BLP-UI sets are MIBL sets and sets in $\DMILPR$ are closed. The reverse inclusion is immediate from \cref{cor:closure-of-mibl-is-dmilpr}. 
\end{proof}


Finally, we prove \cref{thm:UpIntisNP}.
\begin{proof}[Proof of \cref{thm:UpIntisNP}.]
	Following the notation in \cref{Def:MIBL}, $S=S^1\cap S^2\cap S^3$. Since $|\mathcal I_F|=0$, $(x,\,y)\in S^2$ if and only if there exists a $\lambda\leq 0$ such that $D^\top \lambda = f$ and $\lambda^\top (g-Cx) = f^\top y$. However, this is same as checking if there exists $(x,\,y)$ such that $(x,\,y)\in S^1\cap S^3 \cap \left\{ (x,\,y): \exists\,\lambda\leq 0,\,D^\top \lambda =f,\, \lambda^\top g - f^\top y \leq \lambda^\top Cx  \right\}$ is non-empty. But this set is a set of linear inequalities, integrality requirements along with {\em exactly one } quadratic inequality. From \citet{del2017mixed}, this problem is in $\mathcal{NP}$.
\end{proof}

By reduction from regular integer programming, we obtain this corollary.

\begin{corollary}
	Bilevel linear programs with rational data integrality constraints only in the upper level is $\mathcal {NP}$-complete.
\end{corollary}

\section{Conclusion}\label{sec:conclusion}

In this paper we give a characterization of the types of sets that are representable by feasible regions of mixed-integer bilevel linear programs. In the case of bilevel linear programs with only continuous variables, the characterization is in terms of the finite unions of polyhedra. In the case of mixed-integer variables, the characterization is in terms of finite unions of generalized mixed-integer linear representable sets. Interestingly, the family of finite unions of polyhedra and the family of finite unions of generalized mixed-integer linear representable sets are both algebras of sets. The parallel between these two algebras suggests that generalized mixed-integer linear representable sets are the ``right'' geometric vocabulary for describing the interplay between projection, integrality, and bilevel structures. We are hopeful that the algebra properties of finite unions of generalized mixed-integer linear representable sets may prove useful in other contexts. 

There remain important algorithmic and computational questions left unanswered in this study. For instance, is there any computational benefit of expressing a bilevel program in terms of disjunctions of (generalized) mixed-integer linear representable sets? Are there problems that are naturally expressed as disjunctions of generalized mixed-integer linear representable sets that can be solved using algorithms for mixed-integer bilevel optimization (such as the recent work by \citet{Fischetti2017,Wang2017}, etc.)? In the case of continuous bilevel linear problems, its equivalence with unions of polyhedra suggests a connection to solving problems over unions of polyhedra, possibly using the methodology of \citet{Balas1996}. The connection between bilevel constraints and linear complementarity also suggest a potential for interplay between the computational methods of both types of problems. 

%
%

\section*{Acknowledgments}

The first and third authors are supported by NSF grant CMMI1452820 and ONR grant N000141812096. The second author thanks the University of Chicago Booth School of Business for its generous financial support.

\bibliographystyle{plainnat}
\bibliography{library}
\end{document}